\documentclass[11pt,reqno]{amsart}
\usepackage{amsmath,mathtools,amsthm,amssymb}
\allowdisplaybreaks[4]
\usepackage{graphicx}
\usepackage{epsfig,epsf,psfrag}
\usepackage{amsfonts}
\usepackage{setspace}
\usepackage{color}
\usepackage[font=small]{caption}
\usepackage[font=footnotesize]{subcaption}
\usepackage[top=1in, bottom=1in, left=1.2in, right=1.2in]{geometry}
\usepackage{afterpage}
\usepackage{url}
\usepackage{isomath}
\usepackage{booktabs}
\usepackage{framed,multirow}
\newtheorem{theorem}{Theorem}[section]

\newtheorem{lemma}[theorem]{Lemma}

\usepackage{algorithm}
\usepackage{algorithmic}
\usepackage{fancyhdr}


\title{DTPFI: A stable algorithm for recovering nonlinear energy potentials in phase field systems}

\author{Tianhao Ni}
\address{School of Mathematical Sciences, Fudan University, Shanghai, 200433, China}
\email{thni@zju.edu.cn}

\author{Jun Lai}
\address{School of Mathematical Sciences, Zhejiang University, Hangzhou, Zhejiang 310027, China}
\email{laijun6@zju.edu.cn}
\keywords{Phase field models, Allen-Cahn equations, Cahn-Hilliard equations, inverse potential problems, parameter identification}

\begin{document}
	\begin{abstract}
This work proposes a Dual Time Phase Field Inversion (DTPFI) method for recovering unknown potential functions in phase field models. The reconstruction is formulated as an optimization problem that minimizes the mismatch between model predictions and observed fields at the final measurement time. 
We prove the differentiability of the measured field with respect to the unknown potential and establish the local convexity of the regularized objective function, thereby ensuring the existence of a local optimal solution. 
For numerical implementation, automatic differentiation is employed to compute gradients, avoiding the expensive  evaluation of analytic gradients.
Extensive numerical experiments demonstrate that DTPFI accurately reconstructs both polynomial and logarithmic potentials and remains robust under measurement noise.
The framework is further extended to inverse problems involving field dependent mobility and joint parameter identification in coupled Cahn-Hilliard-Allen-Cahn systems.
\end{abstract}
	\maketitle

\section{Introduction}
Phase field modeling employs continuous field variables to describe microstructural evolution in materials science \cite{wang1995microstructural,rubin1999three,khachaturyan2013theory}. It has become an important tool for studying phenomena driven by external fields, including temperature \cite{liu2012micro}, stress-strain \cite{hu2007phase}, and electromagnetic fields \cite{sun2023phase}. Beyond its broad applications in materials science, this approach has gradually expanded into areas such as biomedicine \cite{lima2014hybrid,travasso2011phase} and earth sciences \cite{wendler2009phase,mikelic2015phase}, providing an effective bridge between microscopic mechanisms and macroscopic simulations. The fundamental governing equations of phase field theory are commonly represented by two classes of partial differential equations: the Allen-Cahn (AC) equation \cite{allen1979microscopic} for non-conserved order parameters and the Cahn-Hilliard (CH) equation \cite{cahn1958free} for conserved order parameters. However, practical applications still face a critical bottleneck: key model parameters, such as chemical potential coefficients and elastic modulus tensors, are often difficult to measure directly. This limitation can substantially reduce the predictive reliability of phase field models for complex evolution processes. Therefore, the goal of this paper is to develop an effective method for recovering nonlinear energy potentials in phase field models from measured data.

It is worth mentioning that forward problems for phase field equations have been extensively studied, leading to a wealth of theoretical and numerical results. In particular, the mathematical well-posedness of the Allen-Cahn and Cahn-Hilliard equations has been thoroughly investigated. For classical polynomial nonlinearities, especially the double-well potential, the existence and uniqueness of weak solutions have been rigorously established \cite{elliott1989nonconforming, evans2022partial}. The analysis has also been extended to logarithmic nonlinear potentials, which in certain cases provide a more physically accurate description of phase separation despite their mathematical singularities \cite{cherfils2000generalized, morocsanu2016well}. Beyond individual equations, well-posedness theories have been developed for coupled phase field systems, including multi-component Cahn-Hilliard systems and coupled Cahn-Hilliard-Allen-Cahn (CH-AC) models \cite{barrett2002finite, barrett2001fully, brunk2021analysis, feng2006fully}.

From a numerical perspective, the high-order spatial derivatives and strong nonlinearities inherent in phase field equations impose stringent demands on discretization schemes. For spatial discretization, high-order techniques such as finite difference methods \cite{chen2002phase}, finite element methods \cite{shen2010numerical}, and Fourier spectral methods \cite{chen1998applications} have been widely used to accurately capture interfacial evolution. Since phase field models satisfy energy dissipation laws, temporal discretization has increasingly focused on constructing efficient, unconditionally energy stable schemes for robust long time simulations. To date, methods such as convex splitting \cite{eyre1998unconditionally}, exponential time differencing (ETD) \cite{fu2022energy}, and the scalar auxiliary variable (SAV) approach \cite{shen2018scalar} have become standard techniques for handling strong nonlinearities while balancing computational efficiency and physical consistency.

In contrast to the rapid development of theories and numerical algorithms for forward problems, research on inverse problems for phase field models remains at an early stage. On the theoretical side, a uniqueness theory for energy functional inversion was established in \cite{brunk2023uniqueness} by using the weak form of the Cahn-Hilliard equation. The identifiability of energy functionals in coupled Cahn-Hilliard-Allen-Cahn systems has also been demonstrated from concentration field data alone \cite{ni2024uniqueness}, with applications to the modeling of high temperature alloys. On the computational side, a spatiotemporal data driven inversion method based on deep learning was proposed in \cite{bao2025pfwnn}. However, these approaches generally require information on time derivatives of field variables. In experimental observations, high precision instantaneous rates of change are difficult to obtain directly, and numerical differentiation of discrete data inevitably amplifies measurement noise, thereby significantly limiting the applicability of such methods.

For inverse problems in which only field measurements at discrete time points are available, existing studies have primarily focused on coefficient identification in linear parabolic equations or on cases where the nonlinear form is prescribed in advance. For instance, monotonic operators have been constructed using fixed point theorems to develop finite element based convergence proofs and numerical schemes \cite{zhang2022identification}. This framework has been further extended to semi-linear parabolic equations for identifying coefficients in nonlinear terms \cite{kaltenbacher2025reconstruction,kaltenbacher2026identification}. In addition, optimal control frameworks have been widely used to reformulate parameter identification as a functional minimization problem, with corresponding analyses of convergence rates for finite element discretizations \cite{kahle2020parameter,jin2021error,jin2023convergence,zhang2022identification}. Nevertheless, the inversion of unknown nonlinear potentials in phase field equations remains largely unexplored, especially when only concentration observations are available and no time derivative information is used.

To address these challenges, this work proposes the Dual Time Phase Field Inversion (DTPFI) method, which recovers nonlinear energy potentials from field measurements at two discrete time points. We formulate the inverse problem as a finite dimensional optimization problem via basis truncation, establish the differentiability of the parameter to state map, prove the existence of minimizers, and derive local strong convexity of the objective functional. Computationally, we develop a practical algorithm that avoids continuous observations and time derivative data by incorporating automatic differentiation for efficient gradient evaluation, and we justify its equivalence to analytical sensitivity based gradients from a computational graph perspective. We validate DTPFI through extensive experiments on logarithmic and polynomial potentials, demonstrating accurate reconstruction under measurement noise as well as effectiveness for field dependent mobility and joint parameter identification in coupled CH-AC systems.

The remainder of this paper is organized as follows. Section 2 introduces the AC and CH phase field equations, together with their well-posedness properties and relevant Sobolev spaces. Section 3 presents the DTPFI formulation, including the regularized objective function and theoretical results on differentiability and local strong convexity. Section 4 derives the discrete gradients and analyzes the consistency between automatic differentiation and analytical gradient computation. Section 5 provides numerical results for logarithmic and polynomial potential inversion, as well as  joint parameter identification in coupled CH-AC systems. Finally, Section 6 summarizes the main results and discusses future research directions.

\section{Mathematical Formulation}

\subsection{Governing equations}
The phase field method is a modeling approach rooted in thermodynamic principles, which characterizes the dynamical behavior of complex systems through spatiotemporal evolution equations of continuous field variables \cite{steinbach2009phase}. Based on the Ginzburg-Landau theory \cite{ginzburg2009theory}, this method introduces an order parameter field $u(x,t)$ to describe microstructural evolution, thereby establishing a system of nonlinear partial differential equations that couples diffusion kinetics and thermodynamic driving forces \cite{chen2002phase}. According to conservation principles, phase field models are primarily classified into two canonical forms: the Allen-Cahn (AC) equation for non-conserved dynamics and the Cahn-Hilliard (CH) equation for conserved processes. Their variational formulations are expressed as:
\begin{align*}
	\text{Allen-Cahn:} & \quad \frac{\partial u}{\partial t} = -M \frac{\delta \mathcal{F}}{\delta u}, \quad (x,t) \in \Omega \times (0,T], \\
	\text{Cahn-Hilliard:} & \quad \frac{\partial u}{\partial t} = M \Delta \frac{\delta \mathcal{F}}{\delta u}, \quad (x,t) \in \Omega \times (0,T], 
\end{align*}
where $M > 0$ is the mobility coefficient and $T > 0$ is the final time. The total free energy functional $\mathcal{F}(u)$ is defined as:
\begin{align}
	\label{equ:energy_func}
	\mathcal{F}(u) = \int_{\Omega} \left[ \frac{\varepsilon^2}{2}|\nabla u|^2 + F(u) \right] \mathrm{d} x,
\end{align}
where $\varepsilon$ characterizes the interface width and $F(u)$ is the bulk energy density. Two commonly used forms of $F(u)$ and their derivatives $f(u) = F'(u)$ are:
\begin{itemize}
\item The polynomial double-well potential, given by
		\begin{align}\label{doubwell}
			F(u) = \frac{1}{4}(u^2-1)^2, 
		\end{align}
with $f(u) = u^3 - u$.

\item  The logarithmic potential, defined as
		\begin{align}\label{logpot}
			F(u) = \frac{\theta_c}{2}(1-u^2) + \frac{\theta}{2} \left[ (1-u)\ln\frac{1-u}{2} + (1+u)\ln\frac{1+u}{2} \right],
		\end{align}
with derivative $f(u) = \frac{\theta}{2} \ln \left( \frac{1+u}{1-u} \right) - \theta_c u$, where $\theta_c$ denotes the critical temperature and $\theta < \theta_c$ denotes the absolute temperature.
\end{itemize}
Substituting the variational derivative $\frac{\delta \mathcal{F}}{\delta u} = -\varepsilon^2 \Delta u + f(u)$ into the governing equations yields the strong forms as:
\begin{align}
	\label{equ:AC1}
	\text{Allen-Cahn:} & \quad \frac{\partial u}{\partial t} = M(\varepsilon^2\Delta u - f(u)), \\
	\label{equ:CH1}
	\text{Cahn-Hilliard:} & \quad \frac{\partial u}{\partial t} = -M\Delta(\varepsilon^2\Delta u - f(u)).
\end{align}
For convenience, we unify \eqref{equ:AC1} and \eqref{equ:CH1} into a general form:
\begin{align}
	\label{equ:general_form}
	\frac{\partial u}{\partial t} = M(-\Delta)^{\alpha}(\varepsilon^2\Delta u - f(u)),
\end{align}
where $\alpha=0$ and $\alpha=1$ correspond to the AC and CH equations, respectively.

\subsection{Well-posedness}
We consider the phase field equations on a unit hypercube $\Omega = [0,1]^d$ subject to periodic boundary conditions (PBCs). Let $L^2_p(\Omega)$ and $H^k_p(\Omega)$ denote the subspaces of $L^2(\Omega)$ and $H^k(\Omega)$ consisting of functions that satisfy the periodic boundary conditions on $\partial\Omega$. For any $k \in \mathbb{N}$, we define the following norms for these spaces:
\begin{align*}
	&\|u\|_{L^2_p(\Omega)} =\|u\|_{L^2(\Omega)} =  \left( \int_{\Omega} |u|^2 \mathrm{d}x \right)^{1/2},\\
	&\|u\|_{H^k_p(\Omega)} =\|u\|_{H^k(\Omega)} =\left( \sum_{|\beta| \leq k} \int_{\Omega} |D^{\beta} u|^2 \mathrm{d}x \right)^{1/2}.
\end{align*}
Furthermore, we define the Bochner space $\mathcal{L}^k = L^2(t_1,t_2; H^k_p(\Omega))$ equipped with the norm $\|u\|_{\mathcal{L}^k} = \left( \int_{t_1}^{t_2} \|u\|_{H^k_p(\Omega)}^2 \mathrm{d}t \right)^{1/2}$, where $t_1 < t_2$ are two distinct time points within $[0,T]$.

The well-posedness of the systems under the initial condition $u(x,0) = u_0(x)$ is summarized as follows:

\begin{lemma}[Well-posedness of the CH equation \cite{miranville2019cahn,cherfils2011cahn}]
\label{lemma:CH}
Consider the CH equation with initial condition $u_0 \in H^2_p(\Omega)$. 
\begin{itemize}
    \item[1.] For the polynomial potential \eqref{doubwell}, there exists a unique solution $u$ to the CH equation such that $u \in L^{\infty}(0,T; H^2_p(\Omega)) \cap L^2(0,T; H^4_p(\Omega))$. 
    \item[2.] For the logarithmic potential \eqref{logpot}, if $-1 < u_0(x) < 1$ a.e., the unique solution to the CH equation satisfies $u \in L^\infty(0, T; H_p^2(\Omega)) \cap C([0, T]; H^{-1}(\Omega))$.
\end{itemize}
\end{lemma}

\begin{lemma}[Well-posedness of the AC equation \cite{ladyzhenskaia1968linear,bertacco2021stochastic}]
\label{lemma:AC}
Consider the AC equation with initial condition $u_0 \in H_p^1(\Omega)$. 
\begin{itemize}
    \item[1.] For the polynomial potential \eqref{doubwell}, there exists a unique solution $u$ to the AC equation such that $u \in L^{\infty}(0,T;L_p^2(\Omega))\cap L^2(0,T;H_p^1(\Omega))$. 
    \item[2.] For the logarithmic potential \eqref{logpot}, assuming $-1 < u_0(x) < 1$ a.e., the unique solution to the AC equation satisfies $u \in C(0,T; L_p^2(\Omega)) \cap L^2(0,T; H_p^2(\Omega))$.
\end{itemize}
\end{lemma}

For the logarithmic potential, the admissible state must remain in the physical interval $(-1,1)$, because both $F(u)$ and $f(u)$ are singular at $u=\pm1$.  For the polynomial potential, although no singular constraint is present, we assume that the phase field remains in a fixed bounded interval $I\subset\mathbb{R}$. This boundedness is natural for phase field variables and ensures that the potential and its derivatives are evaluated only on the physically relevant range used in the inverse reconstruction.

\section{Inverse Problems of phase field Equations}
\subsection{Inverse problem}
In this section, we formulate the inverse problem of reconstructing the nonlinear potential function in phase field models. The inversion framework uses field measurements at two discrete time points. Specifically, given the dual time measurement data $\{u_{\text{obs}}(\cdot,t_1), u_{\text{obs}}(\cdot,t_2)\}$ with $0 < t_1 < t_2 \leq T$, our goal is to recover the potential $f$. The reconstruction is formulated as the following constrained optimization problem:
\begin{equation}
	\label{equ:opt1}
	\begin{aligned}
		\min\limits_{f} \mathcal{J}(f) = &\|u_f(\cdot,t_2) - u_{\text{obs}}(\cdot,t_2)\|_{L^2(\Omega)}^2 + \mathcal{R}(f), 
	\end{aligned}
\end{equation}
where $\mathcal{R}(f)$ is a regularization term introduced to mitigate the ill-posedness of the inverse problem. This term is typically defined as the squared $L^2$ norm of $f$, i.e., $\mathcal{R}(f) = \beta\|f\|_{L^2}^2$, with $\beta > 0$ denoting the regularization parameter. The state variable $u_f$ is governed by the following system:
\begin{equation}
	\label{equ:general_phase}
	\left\{
	\begin{aligned}
		&\frac{\partial u_f}{\partial t} = M(-\Delta)^{\alpha}\left(\varepsilon^2\Delta u_f - f(u_f)\right) && \text{in } \Omega \times (t_1, t_2], \\
		&u_f(\cdot,t_1) = u_{\text{obs}}(\cdot,t_1) && \text{in } \Omega, \\
		&u_f \text{ satisfies PBCs} && \text{on } \partial\Omega \times [t_1, t_2].
	\end{aligned}
	\right.
\end{equation}

To address the infinite dimensional nature of the optimization problem \eqref{equ:opt1}, we adopt a parametric approach by expanding the potential function into a series of basis functions. Specifically, we assume that the target potential $f(s)$ can be represented as:
\begin{align*}
	f(s) = \sum_{n=0}^{\infty} a_n \phi_n(s), \quad s\in I,
\end{align*}
where $\{\phi_n\}_{n=0}^\infty$ constitutes a complete basis system in the interval $I\subset\mathbb{R}$. By truncating the infinite series to the first $N+1$ terms, we obtain the finite dimensional approximation:
\begin{align*}
	f_N(s) = \sum_{n=0}^{N} a_n \phi_n(s).
\end{align*}
This transformation converts the original problem \eqref{equ:opt1} into the identification of the finite coefficient vector $\mathbf{a} = (a_0, \dots, a_N) \in \mathbb{R}^{N+1}$:
\begin{equation}
	\label{equ:opt2}
	\begin{aligned}
		\min\limits_{\mathbf{a} \in \mathbb{R}^{N+1}} \mathcal{J}_N(\mathbf{a}) = & \|u_{\mathbf{a}}(\cdot,t_2) - u_{\text{obs}}(\cdot,t_2)\|_{L^2(\Omega)}^2 +  \mathcal{R}(\mathbf{a}), 
	\end{aligned}
\end{equation}
where the regularization term is specified as $\mathcal{R}(\mathbf{a}) = \beta\|\mathbf{a}\|_2^2$, and the state $u_{\mathbf{a}}$ satisfies the discretized phase field equation:
\begin{equation}
	\label{equ:ua}
	\left\{
	\begin{aligned}
		&\frac{\partial u_{\mathbf{a}}}{\partial t} = M(-\Delta)^{\alpha}\left(\varepsilon^2\Delta u_{\mathbf{a}} - \sum_{n=0}^N a_n \phi_n(u_{\mathbf{a}})\right) && \text{in } \Omega \times (t_1, t_2], \\
		&u_{\mathbf{a}}(\cdot,t_1) = u_{\text{obs}}(\cdot,t_1) && \text{in } \Omega, \\
		&u_{\mathbf{a}} \text{ satisfies PBCs} && \text{on } \partial\Omega \times [t_1, t_2].
	\end{aligned}
	\right.
\end{equation}

\subsection{Theoretical Analysis} 
In this subsection, we provide a theoretical foundation for the optimization problem \eqref{equ:opt2}. We first recall the well-posedness results for fourth order linear parabolic equations, as summarized in \cite{ni2024uniqueness}:

\begin{lemma}
	\label{lemma:well_posedness_four_order}
	Consider the following initial-boundary value problem for $u(x,t)$:
	\begin{equation}
		\label{equ:linear CH}
		\left\{
		\begin{aligned}
			&u_{t} + a\Delta^2 u - \Delta (b(x,t)u) = p(x,t) && \text{in }  \Omega \times (t_1,t_2], \\
			&u(x,t_1) = \psi(x) && \text{in } \Omega, \\
			&u \text{ satisfies PBCs} && \text{on } \partial\Omega \times [t_1,t_2].
		\end{aligned}
		\right.
	\end{equation}
	Assume $a > 0$ and $b \in L^{\infty}(t_1,t_2; H_p^2(\Omega))$. For any source term $p \in L^{\infty}(t_1,t_2; L_p^{2}(\Omega))$ and initial value $\psi \in H_p^2(\Omega)$, there exists a unique solution $u \in L^{\infty}(t_1,t_2; H_p^2(\Omega)) \cap L^2(t_1,t_2; H^4_p(\Omega))$. Furthermore, the solution satisfies the stability estimate:
	\begin{align*}
		\|u\|_{L^{\infty}(t_1,t_2; H_p^2(\Omega))} + \|u\|_{L^2(t_1,t_2; H^4_p(\Omega))} \leq C \left( \|\psi\|_{H_p^2(\Omega)} + \|p\|_{L^{\infty}(t_1,t_2; L_p^{2}(\Omega))} \right),
	\end{align*}
	where the constant $C$ depends on $\Omega, T, a$, and $\|b\|_{L^{\infty}(t_1,t_2; H_p^2(\Omega))}$.
\end{lemma}

We now establish the differentiability of the state variable $u_{\mathbf{a}}$ with respect to the parameter vector $\mathbf{a}$, which ensures the existence of gradients for the objective function $\mathcal{J}_N$.

\begin{theorem}\label{thm:differentiability}
	Assume that the basis functions $\{\phi_n\}_{n=0}^{N}$ are in $C^2(I)$, where $I$ is the range of $u_{\mathbf{a}}$, and $u_{\text{obs}}(\cdot,t_1)\in H_p^2(\Omega)$. Then the state variable $u_{\mathbf{a}}$ is twice continuously differentiable with respect to $\mathbf{a}$. The sensitivity functions $v_i = \partial u_{\mathbf{a}} / \partial a_i$ and $w_{ij} = \partial^2 u_{\mathbf{a}} / \partial a_i \partial a_j$ are the unique solutions to the following systems:
	\begin{itemize}
		\item[1.] \textbf{First order sensitivity equation:}
		\begin{equation}
			\label{equ:vi_full}
			\left\{
			\begin{aligned}
				&\frac{\partial v_i}{\partial t} = M(-\Delta)^{\alpha}\left[\varepsilon^2\Delta v_i - \sum_{n=0}^{N} a_{n}\phi_n'(u_{\mathbf{a}})v_i - \phi_i(u_{\mathbf{a}})\right] && \text{in } \Omega\times (t_1,t_2], \\
				&v_i(x,t_1) = 0 && \text{in } \Omega, \\
				& v_i \text{ satisfies PBCs} && \text{on } \partial\Omega\times[t_1,t_2].
			\end{aligned}
			\right.
		\end{equation}
		\item[2.] \textbf{Second order sensitivity equation:}
		\begin{equation}
			\label{equ:wij_full}
			\left\{
			\begin{aligned}
				&\frac{\partial w_{ij}}{\partial t} = M(-\Delta)^{\alpha}\left[\varepsilon^2\Delta w_{ij}-\sum\limits_{n=0}^N\left(a_{n}\phi_{n}''(u_{\mathbf{a}})v_iv_j+a_{n}\phi_n'(u_{\mathbf{a}})w_{ij}\right)\right] \\
				&\quad \quad \quad -M(-\Delta)^{\alpha}\left[\phi_i'(u_{\mathbf{a}})v_j+\phi_j'(u_{\mathbf{a}})v_i\right] && \text{in } \Omega\times [t_1,t_2], \\
				&w_{ij}(x,t_1) = 0 && \text{in } \Omega, \\
				&w_{ij}\text{ satisfies PBCs} && \text{on } \partial\Omega\times[t_1,t_2].
			\end{aligned}
			\right.
		\end{equation}
	\end{itemize}
\end{theorem}

\begin{proof}
	We provide the proof for the Cahn-Hilliard case ($\alpha=1$). The Allen-Cahn case ($\alpha=0$) can be treated analogously using the well-posedness of second order parabolic equations as in \cite{evans2022partial}.
	
	\textbf{Step 1: Continuous dependence on $\mathbf{a}$.} Let $\mathbf{a}_1, \mathbf{a}_2 \in \mathbb{R}^{N+1}$ be two parameter vectors, and $u_{\mathbf{a}_1}, u_{\mathbf{a}_2}$ be their corresponding solutions in $L^\infty(t_1, t_2; H_p^2(\Omega))$ as shown in Lemma \ref{lemma:CH}. The difference $v_{\mathbf{a}} = u_{\mathbf{a}_1} - u_{\mathbf{a}_2}$ satisfies:
	\begin{equation}
		\label{equ:va_diff_system}
		\left\{
		\begin{aligned}
			&\frac{\partial v_{\mathbf{a}}}{\partial t} + M\varepsilon^2\Delta^2 v_{\mathbf{a}} - M \Delta \left( \sum_{n=0}^{N} a_{1,n} \bar{\phi}_n' v_{\mathbf{a}} \right) \\ &= M \Delta \left( \sum_{n=0}^{N} (a_{1,n} - a_{2,n}) \phi_n(u_{\mathbf{a}_2}) \right) && \text{in } \Omega \times (t_1, t_2], \\
			&v_{\mathbf{a}}(x, t_1) = 0 && \text{in } \Omega,\\
				& v_{\mathbf{a}} \text{ satisfies PBCs} && \text{on } \partial\Omega\times[t_1,t_2].
		\end{aligned}
		\right.
	\end{equation}
	where 
	\begin{align*}
		\bar{\phi}_n' = \int_{0}^{1} \phi_n'(u_{\mathbf{a}_2} + \tau v_{\mathbf{a}}) \mathrm{d}\tau.
	\end{align*} 
	Let $b(x,t) = \sum a_{1,n} \bar{\phi}_n'$ and $p(x,t) = M \Delta \sum (a_{1,n}-a_{2,n}) \phi_n(u_{\mathbf{a}_2})$. Since $\phi_n \in C^4(I)$ and $u_{\mathbf{a}_k} \in L^\infty(t_1, t_2; H_p^2(\Omega))$, it follows that $b \in L^{\infty}(t_1, t_2; H_p^2(\Omega))$. Applying the stability estimate in Lemma \ref{lemma:well_posedness_four_order} to the difference system \eqref{equ:va_diff_system}, we have:
	\begin{align*}
			&\|v_{\mathbf{a}}\|_{L^{\infty}(t_1, t_2; H_p^2(\Omega))} + \|v_{\mathbf{a}}\|_{L^2(t_1,t_2; H^4_p(\Omega))} \\
			\leq & C M \sup_{t \in [t_1, t_2]} \left\| \sum_{n=0}^{N} \delta a_n \phi_n(u_{\mathbf{a}_2}(\cdot, t)) \right\|_{H_p^2(\Omega)} \\
			\leq & C M \sum_{n=0}^{N} |\delta a_n| \cdot \|\phi_n(u_{\mathbf{a}_2})\|_{L^{\infty}(t_1, t_2; H_p^2(\Omega))} \\
			\leq & C M \underbrace{\left( \sum_{n=0}^{N} |\delta a_n|^2 \right)^{1/2}}_{\|\mathbf{a}_1 - \mathbf{a}_2\|_2} \left( \sum_{n=0}^{N} \|\phi_n(u_{\mathbf{a}_2})\|_{L^{\infty}(t_1, t_2; H_p^2(\Omega))}^2 \right)^{1/2} \\
			= & C' \|\mathbf{a}_1 - \mathbf{a}_2\|_2,
	\end{align*}
	where $\delta a_n = a_{1,n} - a_{2,n}$. The constant $C'$ depends on $M, \varepsilon, T, \|\mathbf{a}_1\|_2$, and the $H_p^2$-norms of the basis functions $\{\phi_n(u_{\mathbf{a}_2})\}_{n=0}^N$ which are bounded since $u_{\mathbf{a}_2} \in L^{\infty}(t_1, t_2; H_p^2(\Omega))$ and $\phi_n \in C^2(I)$. This chain of inequalities implies that the state variable $u_{\mathbf{a}}$ is Lipschitz continuous with respect to the parameter vector $\mathbf{a}$.
	
	\textbf{Step 2: First-order differentiability.} To prove the differentiability with respect to $a_i$, let $\mathbf{a}_1 = \mathbf{a} + \delta a_i \mathbf{e}_i$ and $\mathbf{a}_2 = \mathbf{a}$, where $\mathbf{e}_i$ denotes the $i$-th unit basis vector in $\mathbb{R}^{N+1}$. We define the difference quotient $v_i^{\delta} = (u_{\mathbf{a}_1} - u_{\mathbf{a}_2})/\delta a_i$. Dividing the difference system \eqref{equ:va_diff_system} by $\delta a_i$, we find that $v_i^{\delta}$ satisfies the following linear fourth order parabolic equation:
	\begin{equation}
		\label{equ:vi_delta_system}
		\left\{
		\begin{aligned}
			&\frac{\partial v_i^{\delta}}{\partial t} + M\varepsilon^2\Delta^2 v_i^{\delta} - M \Delta \left( \sum_{n=0}^{N} a_{1,n} \bar{\phi}_n' v_i^{\delta} \right) = M \Delta \phi_i(u_{\mathbf{a}}) && \text{in } \Omega \times (t_1, t_2], \\
			&v_i^{\delta}(x, t_1) = 0 && \text{in } \Omega, \\
			&v_i^{\delta} \text{ satisfies PBCs} && \text{on } \partial\Omega \times [t_1, t_2],
		\end{aligned}
		\right.
	\end{equation}
	where $\bar{\phi}_n'$ is the integrated derivative defined in Step 1. As $\delta a_i \to 0$, we have $\mathbf{a}_1 \to \mathbf{a}_2$ in $\mathbb{R}^{N+1}$, and from the Lipschitz continuity proved in Step 1, $u_{\mathbf{a}_1} \to u_{\mathbf{a}_2}$ in $L^{\infty}(t_1, t_2; H_p^2(\Omega))$. Consequently, $\bar{\phi}_n' \to \phi_n'(u_{\mathbf{a}_2})=\phi_n'(u_{\mathbf{a}})$ in $L^{\infty}(t_1, t_2; H_p^2(\Omega))$. Following the same arguments as in Step 1, there exists a unique solution $v \in L^{\infty}(t_1,t_2; H_p^2(\Omega)) \cap \mathcal{L}^4$ to the system \eqref{equ:vi_delta_system}.
	\textbf{Step 3: Second order differentiability.} To establish second order differentiability, let $\mathbf{a}_1 = \mathbf{a} + \delta a_j \mathbf{e}_j$. We define the difference quotient for the sensitivity function as $w_{ij}^{\delta} = (v_i(\mathbf{a}_1) - v_i(\mathbf{a}))/\delta a_j$. By subtracting the first order sensitivity equations \eqref{equ:vi_full} for $\mathbf{a}_1$ and $\mathbf{a}_2$, we find that $w_{ij}^{\delta}$ satisfies:
	\begin{equation}
		\label{equ:wij_delta_system_refined}
		\frac{\partial w_{ij}^{\delta}}{\partial t} + M\varepsilon^2\Delta^2 w_{ij}^{\delta} - M \Delta \left( b_{\delta}(x,t) w_{ij}^{\delta} \right) =M \Delta p_{\delta}(x,t),
	\end{equation}
	with $w_{ij}^{\delta}(x, t_1) = 0$, where the coefficient $b_{\delta}$ and the source term $p_{\delta}$ are given by:
\begin{align*}
	b_{\delta}(x,t) &= \sum_{n=0}^{N} a_{1,n} \phi_n'(u_{\mathbf{a}_1}), \\
	p_{\delta}(x,t) &=  \sum_{n=0}^{N} a_n \left( \frac{\phi_n'(u_{\mathbf{a}_1}) - \phi_n'(u_{\mathbf{a}})}{\delta a_j} \right) v_i(\mathbf{a}) + \phi_j'(u_{\mathbf{a}_1}) v_i(\mathbf{a}) + \frac{\phi_i(u_{\mathbf{a}_1}) - \phi_i(u_{\mathbf{a}})}{\delta a_j} .
\end{align*}
	
	First, we examine the regularity of $b_{\delta}$. Since $u_{\mathbf{a}_1} \in L^{\infty}(t_1, t_2; H_p^2(\Omega))$ and $\phi_n \in C^4(I)$, the composition $\phi_n'(u_{\mathbf{a}_1})$ also resides in $L^{\infty}(t_1, t_2; H_p^2(\Omega))$. Thus, $b_{\delta} \in L^{\infty}(t_1, t_2; H_p^2(\Omega))$, which satisfies the coefficient requirement of Lemma \ref{lemma:well_posedness_four_order}.
	
	Next, we show that $M \Delta p_{\delta} \in L^{\infty}(t_1, t_2; L_p^2(\Omega))$. Let $v_j^{\delta} = (u_{\mathbf{a}_1} - u_{\mathbf{a}})/\delta a_j$. Using the mean value theorem, the term in $p_{\delta}$ can be rewritten as:
	\begin{align*}
		p_{\delta}  = \sum_{n=0}^{N} a_n \left( \bar{\phi}_n'' v_j^{\delta} \right) v_i(\mathbf{a}) + \phi_j'(u_{\mathbf{a}_1}) v_i(\mathbf{a}) + \bar{\phi}_i' v_j^{\delta},
	\end{align*}
	where $\bar{\phi}''$ and $\bar{\phi}'$ are defined via the mean value theorem as:
	\begin{equation}
		\label{equ:integrated_derivatives}
		\bar{\phi}_n'' = \int_{0}^{1} \phi_n''(u_{\mathbf{a}} + \tau(u_{\mathbf{a}_1} - u_{\mathbf{a}})) \mathrm{d}\tau, \quad \text{and} \quad \bar{\phi}_i' = \int_{0}^{1} \phi_i'(u_{\mathbf{a}} + \tau(u_{\mathbf{a}_1} - u_{\mathbf{a}})) \mathrm{d}\tau.
	\end{equation}
Since $v_i, v_j^{\delta} \in L^{\infty}(t_1, t_2; H_p^2(\Omega))$ and $\phi_n \in C^4(I)$, each product term belongs to $L^{\infty}(t_1, t_2; H_p^2(\Omega))$. Consequently,  $M\Delta p_{\delta} \in L^{\infty}(t_1, t_2; L_p^2(\Omega))$.
	
	Applying the stability estimate in Lemma \ref{lemma:well_posedness_four_order} to \eqref{equ:wij_delta_system_refined}, we have:
	\begin{align*}
		\|w_{ij}^{\delta}\|_{L^{\infty}(t_1, t_2; H_p^2(\Omega))} + \|w_{ij}^{\delta}\|_{L^2(t_1,t_2; H^4_p(\Omega))} \leq C \|M\Delta p_{\delta}\|_{L^{\infty}(t_1, t_2; L_p^2(\Omega))} \leq C',
	\end{align*}
	where $C'$ is independent of $\delta a_j$. As $\delta a_j \to 0$, $b_{\delta} \to \sum a_n \phi_n'(u_{\mathbf{a}})$ and $M\Delta p_{\delta}$ converges to the source term in \eqref{equ:wij_full}. Thus, $w_{ij}^{\delta}$ converges to the unique solution $w_{ij}$ of \eqref{equ:wij_full}. This proves that $u_{\mathbf{a}}$ is twice Fréchet differentiable.
\end{proof}

Theorem \ref{thm:differentiability} establishes that the objective function \eqref{equ:opt2} is continuously differentiable with respect to the parameter vector $\mathbf{a}$. Given the coercivity of $\mathcal{J}_N$ with the standard $\ell^2$ regularization term, the existence of an optimal solution follows from the Weierstrass extreme value theorem \cite{gopfert1973mathematische}.

\begin{theorem}[Existence of optimal solution]
	The optimization problem \eqref{equ:opt2} admits at least one minimizer $\mathbf{a}^{*} \in \mathbb{R}^{N+1}$.
\end{theorem}

We next demonstrate that the objective function is strictly convex in a neighborhood of the minimizer. This property ensures the local uniqueness of the solution and the convergence of gradient based optimization algorithms.

\begin{theorem}[Local strong convexity]
	Consider the optimization problem \eqref{equ:opt2} with the $\ell^2$ regularization term $\mathcal{R}(\mathbf{a}) = \beta\|\mathbf{a}\|_2^2$. Suppose the parameter estimate $\mathbf{a}$ lies within a $\rho_1$-neighborhood of the optimal solution $\mathbf{a}^{*}$, i.e., $\|\mathbf{a} - \mathbf{a}^{*}\|_2 \leq \rho_1$, and the corresponding residual satisfies:
	\begin{align}
		\|u_{\mathbf{a}}(\cdot,t_2) - u_{\text{obs}}(\cdot,t_2)\|_{L^2(\Omega)} \leq \rho_2.
	\end{align}
	Then, for a sufficiently large regularization weight $\beta > C\rho_2$, where $C$ is a constant determined by the stability of the second order sensitivity functions, the objective function $\mathcal{J}_N$ is strongly convex within this neighborhood.
\end{theorem}

\begin{proof}
	The Hessian matrix of the objective function $\mathcal{J}_N(\mathbf{a})$ is given by the sum of three terms:
	\begin{align*}
		H(\mathcal{J}_N) = \underbrace{2 \int_{\Omega} \mathbf{v}(\cdot,t_2) \mathbf{v}^T(\cdot,t_2) \mathrm{d}\mathbf{x}}_{G} + \underbrace{2 \int_{\Omega} (u_{\mathbf{a}}(\cdot,t_2) - u_{\text{obs}}(\cdot,t_2)) \mathbb{W}(\cdot,t_2) \mathrm{d}\mathbf{x}}_{E} + 2\beta I,
	\end{align*}
	where $\mathbf{v} = (v_0, \dots, v_N)^T$ is the vector of first-order sensitivities, $\mathbb{W} = (w_{ij})$ is the matrix of second order sensitivities, and $2\beta I$ is the Hessian of the regularization term $\beta \|\mathbf{a}\|_2^2$.
	
	We analyze the positive definiteness of $H(\mathcal{J}_N)$ by considering the quadratic form $\mathbf{h}^T H(\mathcal{J}_N) \mathbf{h}$ for an arbitrary non-zero vector $\mathbf{h} \in \mathbb{R}^{N+1}$:
	\begin{enumerate}
		\item The first term $G$ is the Gauss-Newton matrix, which is positive semi-definite: $\mathbf{h}^T G \mathbf{h} = 2 \|\mathbf{h}^T \mathbf{v}(\cdot,t_2)\|_{L^2(\Omega)}^2 \geq 0$.
		\item For the second term $E$, the quadratic form is 
		\begin{align*}
			\mathbf{h}^T E \mathbf{h} = 2 \int_{\Omega} (u_{\mathbf{a}}(\cdot,t_2) - u_{\text{obs}}(\cdot,t_2)) \left(\sum_{i,j=0}^N w_{ij} h_i h_j\right) \mathrm{d}\mathbf{x}.
		\end{align*}
		Let $\mathcal{W}_h = \sum_{i,j=0}^N w_{ij}(\cdot,t_2) h_i h_j$. Based on the stability estimate from Lemma \ref{lemma:well_posedness_four_order}, there exists a constant $C > 0$ such that $\|\mathcal{W}_h\|_{L^2(\Omega)} \leq C \|\mathbf{h}\|_2^2$. Applying the Cauchy-Schwarz inequality, we have:
		\begin{align*}
			|\mathbf{h}^T E \mathbf{h}| \leq 2 \|u_{\mathbf{a}} - u_{\text{obs}}\|_{L^2(\Omega)} \|\mathcal{W}_h\|_{L^2(\Omega)} \leq 2 \rho_2 C \|\mathbf{h}\|_2^2.
		\end{align*}
		This implies $E \succeq -2C\rho_2 I$.
		\item The regularization term contributes $2\beta I$, which is strictly positive definite for $\beta > 0$.
	\end{enumerate}
	
	Combining these bounds, we obtain:
	\begin{align*}
		H(\mathcal{J}_N) = G + E + 2\beta I \succeq 0 - 2C\rho_2 I + 2\beta I = 2(\beta - C\rho_2) I.
	\end{align*}
	By choosing $\beta > C\rho_2$, the eigenvalue of the Hessian is strictly positive, i.e., $H(\mathcal{J}_N) \succ 0$. Thus, $\mathcal{J}_N(\mathbf{a})$ is strongly convex in the specified $\rho_1$-neighborhood.
\end{proof}

In practice, the regularization parameter $\beta$ balances fidelity to the measured data and stability of the recovered potential. A very small $\beta$ may lead to overfitting of noisy observations and oscillatory reconstructed coefficients, whereas an overly large $\beta$ can over smooth the solution and bias the recovered potential toward the chosen basis prior. Guided by the local convexity condition above, $\beta$ should be large enough to control the residual dependent Hessian perturbation, but not so large that the data misfit term is dominated by the penalty. In the numerical experiments, we select $\beta$ empirically by choosing the smallest value that yields stable convergence and reconstruction errors.

\section{Optimization Scheme}
\subsection{Discretization}

To solve the inverse problem \eqref{equ:opt2} numerically, we employ a gradient-based optimization algorithm. In the continuous setting, the derivative of the objective function with respect to each basis coefficient $a_i$ is given by:
\begin{equation}
	\label{equ:continuous_gradient}
	\frac{\partial \mathcal{J}_{N}}{\partial a_i} = 2\int_{\Omega}(u_{\mathbf{a}}(\mathbf{x},t_2) - u_{\text{obs}}(\mathbf{x},t_2))v_i(\mathbf{x},t_2)\mathrm{d}\mathbf{x} + 2\beta a_i,  
\end{equation}
where the first order sensitivity variable $v_i = \partial u_{\mathbf{a}} / \partial a_i$ satisfies the linear parabolic system \eqref{equ:vi_full}.

To evaluate this gradient at the discrete level, we perform both spatial and temporal discretizations. Let $N_s$ be the number of uniformly distributed spatial measurement points $\{\mathbf{x}^{(n)}\}_{n=1}^{N_s}$, and let $N_t$ denote the total number of time steps in the interval $[t_1, t_2]$. We denote $u_{N_t}$ and $v_{i, N_t}$ as the numerical approximations of the state variable and the sensitivity function at the final time $t_2$, respectively. The discrete objective function is then formulated as:
\begin{align}
	\label{equ:discrete_loss_func}
	\mathcal{J}_{N,N_s} = \frac{1}{N_s} \sum\limits_{n=1}^{N_s} \left[u_{N_t}(\mathbf{x}^{(n)})-u_{\text{obs}}(\mathbf{x}^{(n)},t_2)\right]^2 + \beta \sum\limits_{i=0}^N a_i^2,
\end{align}
and the corresponding fully discrete gradient is calculated as:
\begin{align}
	\label{equ:discrete_gradient}
	\frac{\partial \mathcal{J}_{N,N_s}}{\partial a_i} = \frac{2}{N_s}\sum_{n=1}^{N_s} \left[ u_{N_t}(\mathbf{x}^{(n)}) - u_{\text{obs}}(\mathbf{x}^{(n)}, t_2) \right] v_{i, N_t}(\mathbf{x}^{(n)}) + 2\beta a_i,
\end{align}
for $i=0, 1, \dots, N$. 
Equation \eqref{equ:discrete_gradient} indicates that the core of the gradient computation lies in obtaining the state variable $u_{N_t}$ and its corresponding sensitivity $v_{i, N_t}$ at the final time step. To this end, we adopt the semi-implicit Fourier spectral method (SIFSM) \cite{chen2002phase} as the forward solver. This approach utilizes spectral expansion in the spatial domain to achieve high-order accuracy, while treating linear terms implicitly in time to ensure unconditional numerical stability. To maintain computational efficiency, nonlinear terms are treated explicitly. For the generalized phase field equation \eqref{equ:general_phase}, the first order accurate semi-implicit scheme in the Fourier space is given by:
\begin{equation*}
	\frac{\{u(t+\Delta t)\}_k - \{u(t)\}_k}{\Delta t} = -M\varepsilon^2 |\mathbf{k}|^{2\alpha+2} \{u(t+\Delta t)\}_k - M|\mathbf{k}|^{2\alpha} \{f(u(t))\}_k,
\end{equation*}
where $\Delta t$ is the time step, $\mathbf{k} = (k_1, \dots, k_d)$ is the wave vector, and $\{\cdot\}_k$ denotes the Fourier transform operator. Rearranging the terms yields the following recurrence formula in the spectral space:
\begin{equation}
	\label{equ:SIMSM}
	\{u(t+\Delta t)\}_k = \frac{\{u(t)\}_k - M |\mathbf{k}|^{2\alpha} \Delta t \{f(u(t))\}_k}{1 + M \varepsilon^2 |\mathbf{k}|^{2\alpha + 2} \Delta t}.
\end{equation}

Under this framework, given the observational data $u_{\text{obs}}(\cdot, t_1)$ as the initial state, the forward evolution of the parameterized variable $u_{\mathbf{a}}$ is computed via the discrete iterative scheme:
\begin{align*}
	&\{u_{\mathbf{a},\ell}\}_k = \frac{\{u_{\mathbf{a},\ell-1}\}_k - M|\mathbf{k}|^{2\alpha}\Delta t \left\{\sum_{n=0}^N a_n\phi_n(u_{\mathbf{a},\ell-1})\right\}_k}{1 + M\varepsilon^2|\mathbf{k}|^{2\alpha+2}\Delta t}, \quad \ell = 1,2,\dots,N_t,\\
	&u_{\mathbf{a},0}(\mathbf{x}) = u_{\text{obs}}(\mathbf{x}, t_1),
\end{align*}
where $\ell$ denotes the time-step index. 

Subsequently, to evaluate the sensitivity variable $v_i$ required for \eqref{equ:discrete_gradient}, we differentiate the above iterative scheme with respect to the coefficient $a_i$. By applying the chain rule, the discrete evolution equation for $v_i$ can be derived as:
\begin{align}
	\label{equ:iter_v}
	\left\{v_{i,\ell}\right\}_k = \frac{\left\{v_{i,\ell-1}\right\}_k - M|\mathbf{k}|^{2\alpha}\Delta t \left\{\phi_i(u_{\mathbf{a},\ell-1}) + \sum_{n=0}^N a_n \phi_n'(u_{\mathbf{a},\ell-1}) v_{i,\ell-1} \right\}_k}{1 + M\varepsilon^2|\mathbf{k}|^{2\alpha+2}\Delta t},
\end{align}
for $\ell = 1,\dots,N_t$. Since the initial state $u_{\mathbf{a},0}$ is fixed by observational data and is independent of the parameters $\mathbf{a}$, the initial condition for the sensitivity evolution vanishes, i.e., $v_{i,0}(\mathbf{x}) \equiv 0$. This discrete recurrence corresponds to solving the continuous first order sensitivity equation \eqref{equ:vi_full} via the SIFSM. By synchronously evolving $u_{\mathbf{a}}$ and $v_i$ until the final time $t_2$, the numerical gradient of the objective function is accurately obtained and then substituted into \eqref{equ:discrete_gradient} for parameter updates.

\subsection{Automatic Gradient Computation via Automatic Differentiation}

In the previous section, we derived the discrete sensitivity equations within the framework of the semi-implicit Fourier spectral method (SIFSM). However, traditional analytical derivation methods often lack flexibility in practical applications. Any modification to the potential function model, the choice of basis functions, or the boundary conditions leads to a complete reconstruction of both the theoretical derivation and the underlying numerical implementation. To overcome this limitation, this section introduces a modular parameter inversion framework that uses automatic differentiation (AD) by unrolling the numerical solver into a computational graph.

Unlike classical physics-informed neural networks (PINNs) \cite{yuan2022pinn,bao2025pfwnn}, which typically treat PDEs as soft penalty constraints within the loss function, our framework integrates the exact numerical iteration logic directly into the computational architecture. By treating the SIFSM scheme as a sequence of differentiable numerical layers, the entire forward simulation is formulated as a fully differentiable mapping from the parameter space to the state variables. This design enables seamless gradient based optimization via backpropagation while strictly preserving the high-order numerical precision and stability of spectral methods. The specific algorithmic workflow is illustrated in Figure \ref{fig:flow}.

\begin{figure}[htbp]
	\centering
	\includegraphics[width=0.95\linewidth]{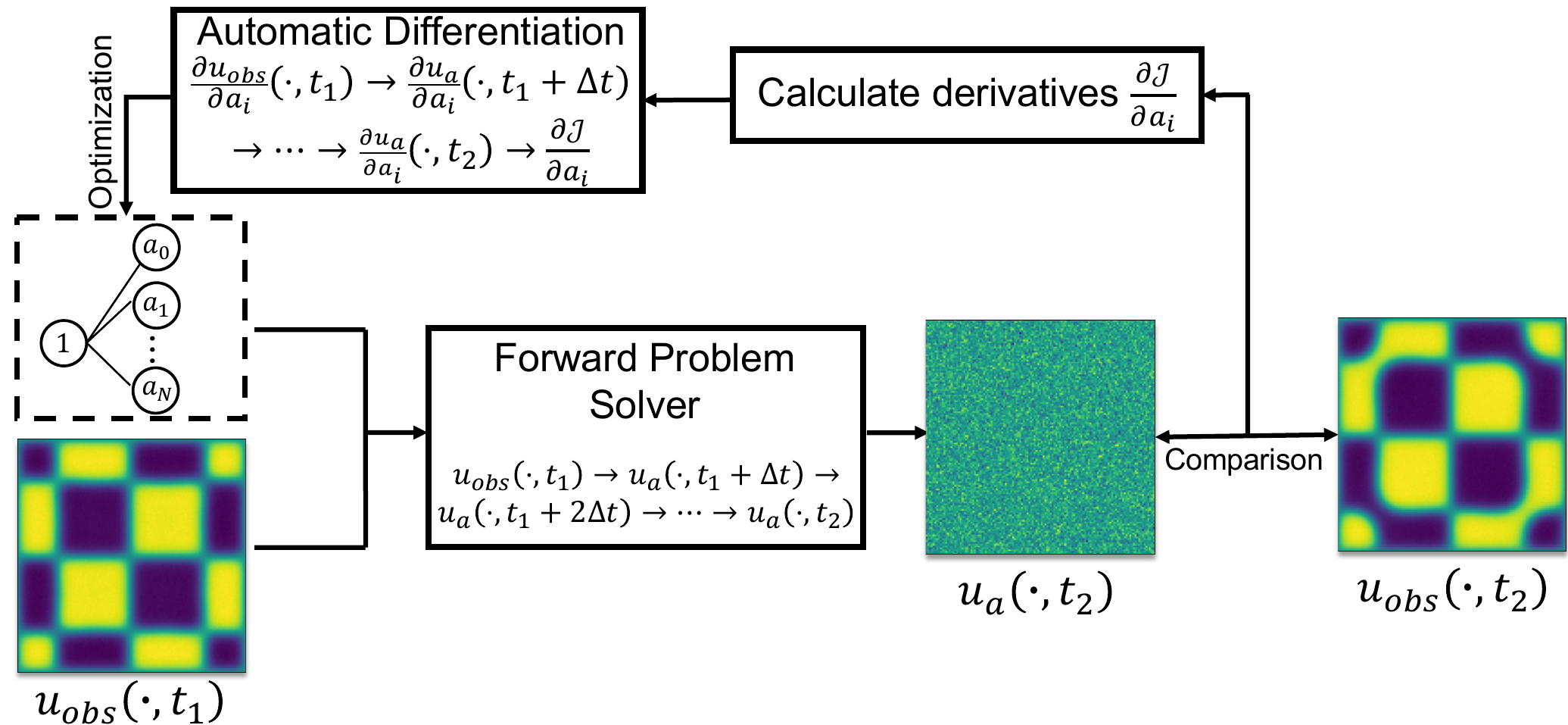}
	\caption{Overview of the DTPFI workflow, illustrating the integration of the physical solver as a differentiable operator within the optimization loop.}
	\label{fig:flow}
\end{figure}

In the forward propagation stage, we use a customized input layer to map the parameters to the computational graph. This layer consists of a single input node connected to an output layer of $N+1$ nodes, without bias terms or nonlinear activations. By fixing the input node to a constant value of $1$, the output of the $i$-th neuron is identical to the $i$-th component of the weight vector $\mathbf{a}$. This mapping transforms the optimization of the basis coefficients $\mathbf{a} = \{a_0, \dots, a_N\}$ into a standard backpropagation process for network weights. Subsequently, the weight vector $\mathbf{a}$ and the initial condition $u_{\mathbf{a},0} = u_{\text{obs}}(\cdot, t_1)$ are passed into cascaded iterative modules (Figure \ref{fig:graph}a). Each module represents a differentiable numerical propagator $\mathcal{F}$ governed by the SIFSM scheme:
\begin{equation}
	\label{equ:F_mapping}
	u_{\mathbf{a},\ell} = \mathcal{F}(u_{\mathbf{a},\ell-1}, \mathbf{a}), \quad \ell = 1, \dots, N_t,
\end{equation}
where the specific operations within $\mathcal{F}$ are defined by the recurrence relation \eqref{equ:SIMSM} and illustrated in Figure \ref{fig:graph}b. After $N_t$ time steps, the final state $u_{\mathbf{a},N_t}$ is used to compute the data loss $L(\mathbf{a})$ and the regularization term $\mathcal{R}(\mathbf{a})$, yielding the total objective function $\mathcal{J}_N(\mathbf{a}) = L(\mathbf{a}) + \mathcal{R}(\mathbf{a})$, with:
\begin{align*}
	L(\mathbf{a}) = \frac{1}{N_s} \sum_{n=1}^{N_s} \left( u_{N_t}(\mathbf{x}^{(n)}) - u_{\text{obs}}(\mathbf{x}^{(n)}, t_2) \right)^2.
\end{align*}

\begin{figure}[htbp]
	\centering
	\includegraphics[width=0.95\linewidth]{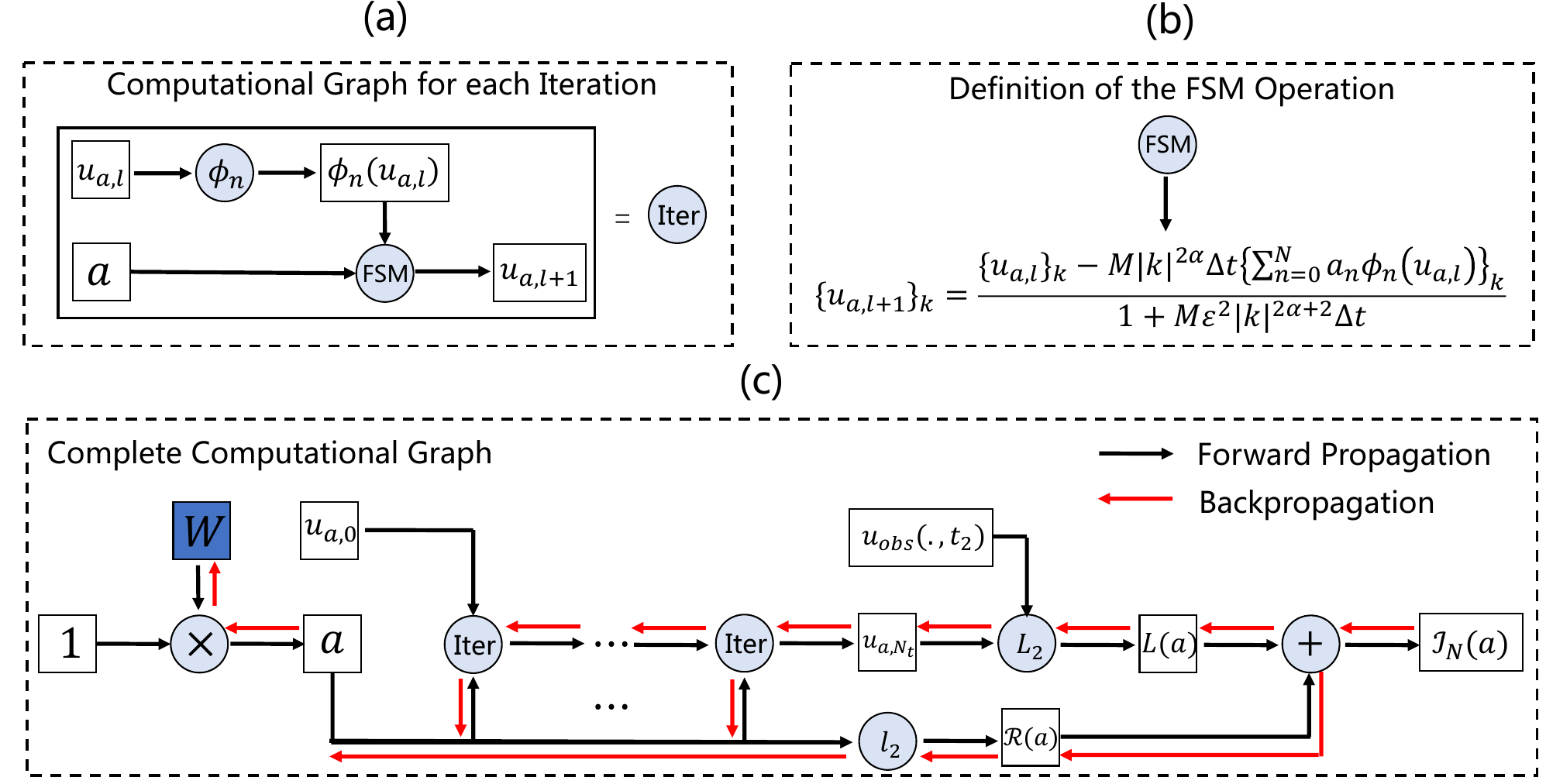}
	\caption{Detailed computational graph of the DTPFI framework: (a) the cascaded time-stepping structure, (b) internal operations of the SIFSM module in spectral space, and (c) the integrated architecture where black and red arrows denote forward and backward propagation, respectively.}
	\label{fig:graph}
\end{figure}

During backward propagation, the gradient flow originates from the objective node. For the regularization term $\mathcal{R}(\mathbf{a})$, the framework directly computes $\partial \mathcal{R} / \partial \mathbf{a}$. For the data loss $L(\mathbf{a})$, as the parameters $\mathbf{a}$ are shared across all temporal modules, the derivative is obtained via the chain rule by accumulating gradients over all paths traversing $\mathbf{a}$. Let the adjoint variable at the $\ell$-th step be $\lambda_\ell = \partial \mathcal{J}_N / \partial u_{\mathbf{a},\ell}$. The total gradient at node $\mathbf{a}$ is the sum of the regularization derivative and the accumulated local gradient branches:
\begin{equation}
	\label{equ:AD_total_grad}
	\frac{\mathrm{d} \mathcal{J}_N}{\mathrm{d} \mathbf{a}} = \frac{\partial \mathcal{R}(\mathbf{a})}{\partial \mathbf{a}} + \sum_{\ell=1}^{N_t} \left( \lambda_\ell \cdot \frac{\partial \mathcal{F}(u_{\mathbf{a},\ell-1}, \mathbf{a})}{\partial \mathbf{a}} \right),
\end{equation}
where $\partial \mathcal{F} / \partial \mathbf{a}$ denotes the partial Jacobian of the operator $\mathcal{F}$ with respect to its explicit parameters.

To show that this AD-based accumulation is mathematically equivalent to the discrete sensitivity equations, we define the local operators as partial derivatives of the mapping $\mathcal{F}$ evaluated at each time step. Specifically, let $G_\ell := \frac{\partial \mathcal{F}}{\partial u}(u_{\mathbf{a},\ell-1}, \mathbf{a})$ be the state transition Jacobian and $H_\ell := \frac{\partial \mathcal{F}}{\partial \mathbf{a}}(u_{\mathbf{a},\ell-1}, \mathbf{a})$ be the partial parameter Jacobian. By applying the total derivative chain rule to \eqref{equ:F_mapping}, the forward sensitivity $v_\ell = \mathrm{d} u_{\mathbf{a},\ell} / \mathrm{d} \mathbf{a}$ evolves linearly as:
\begin{align*}
	v_\ell = G_\ell v_{\ell-1} + H_\ell. 
\end{align*}
With $v_0 = 0$, the sensitivity at the final state $N_t$ expands to $v_{N_t} = \sum_{\ell=1}^{N_t} \left( \prod_{k=\ell+1}^{N_t} G_k \right) H_\ell$. In backward mode, the adjoint variables satisfy $\lambda_{\ell-1} = \lambda_\ell G_\ell$, which yields:
\begin{align*}
	\lambda_\ell = \frac{\partial \mathcal{J}_N}{\partial u_{\mathbf{a},N_t}} \left( \prod_{k=\ell+1}^{N_t} G_k \right).    
\end{align*}
Substituting these into the accumulation formula \eqref{equ:AD_total_grad} (omitting $\mathcal{R}$ for brevity) gives:
\begin{equation}
	\sum_{\ell=1}^{N_t} \lambda_\ell H_\ell = \frac{\partial \mathcal{J}_N}{\partial u_{\mathbf{a},N_t}} \left[ \sum_{\ell=1}^{N_t} \left( \prod_{k=\ell+1}^{N_t} G_k \right) H_\ell \right] \equiv \frac{\partial \mathcal{J}_N}{\partial u_{\mathbf{a},N_t}} v_{N_t}.
\end{equation}
This identity confirms that the AD calculated gradient is precisely the total derivative of the original optimization problem.

Finally, we verify that the operators $G_\ell$ and $H_\ell$ within the AD framework align perfectly with the SIFSM-based sensitivity equations. Direct differentiation of the SIFSM scheme in Fourier space yields:
\begin{equation}
	\{H_{i,\ell}\}_k = \frac{- M|\mathbf{k}|^{2\alpha}\Delta t \{\phi_i(u_{\mathbf{a},\ell-1})\}_k}{1 + M\varepsilon^2|\mathbf{k}|^{2\alpha+2}\Delta t},
\end{equation}
and for the state transition term $G_\ell v_{i,\ell-1}$:
\begin{equation}
	\{G_\ell v_{i,\ell-1}\}_k = \frac{\{v_{i,\ell-1}\}_k - M|\mathbf{k}|^{2\alpha}\Delta t \left\{ \sum_{n=0}^N a_n \phi_n'(u_{\mathbf{a},\ell-1}) v_{i,\ell-1} \right\}_k}{1 + M\varepsilon^2|\mathbf{k}|^{2\alpha+2}\Delta t}.
\end{equation}
Substituting $H_{i,\ell}$ and $G_\ell v_{i,\ell-1}$ into the sensitivity evolution $v_{i,\ell} = G_\ell v_{i,\ell-1} + H_{i,\ell}$ exactly recovers the discrete sensitivity scheme \eqref{equ:iter_v}. This consistency shows that the AD based computational graph preserves the numerical structure of the SIFSM throughout the inversion process. The resulting Dual-Time Phase Field inversion(DTPFI) procedure updates the parameters using the Adam optimizer \cite{kingma2014adam}, as summarized in Algorithm \ref{alg}.

\begin{algorithm}[H]
\caption{DTPFI: Dual-Time Phase Field Inversion}
\label{alg}
\begin{algorithmic}[1]
\REQUIRE Measurement data $\{u_{obs}(\mathbf{x}^{(n)},t_1), u_{obs}(\mathbf{x}^{(n)},t_2)\}_{n=1}^{N_s}$, initial parameters $\mathbf{a}_0$, regularization parameter $\beta$, learning rate $\eta$, Adam hyperparameters $\beta_1, \beta_2$.
\ENSURE Optimized parameters $\mathbf{a}^*$.
\STATE Initialize momentum vectors $\mathbf{m}_0 = \mathbf{0}, \mathbf{v}_0 = \mathbf{0}$.
\WHILE{not converged}
    \STATE Evaluate potential function $f_{N}(u) = \sum_{n=0}^N a_n \phi_n(u)$.
    \STATE Compute state $u_{\mathbf{a}}(\cdot, t_2)$ by solving Equation \eqref{equ:ua} starting from $u_{obs}(\cdot, t_1)$.
    \STATE Calculate total objective function $\mathcal{J}_{N,N_s}(\mathbf{a})$ as \eqref{equ:discrete_loss_func}.
    \STATE Compute gradient $\mathbf{g} = \nabla_{\mathbf{a}} \mathcal{J}_{N,N_s}(\mathbf{a})$ using automatic differentiation.
    \STATE Update $\mathbf{m}$ and $\mathbf{v}$ using Adam update rules.
    \STATE Update parameters: $\mathbf{a} \leftarrow \mathbf{a} - \eta \cdot \text{Adam}(\mathbf{g}, \mathbf{m}, \mathbf{v})$.
\ENDWHILE
\end{algorithmic}
\end{algorithm}

\section{Numerical Experiments}
In this section, we test the effectiveness of the proposed DTPFI algorithm through numerical experiments involving different phase field equations and potential functions. The measurement data are generated using the semi-implicit Fourier spectral method described above and are perturbed with prescribed noise. Throughout the numerical experiments, all functions are assumed to satisfy periodic boundary conditions.

To approximate the chemical potential energy density function $f$, we use the $n$-th order Chebyshev polynomials $\phi_n$ as basis functions and retain the first 10 terms. For the inverse problem, the computational domain $\Omega$ is discretized by uniformly distributed grid points with 128 nodes in each spatial dimension. The observational data consist of the full-field distributions $u(\cdot,T/2)$ and $u(\cdot,T)$ at the intermediate time $T/2$ and final time $T$, respectively. To evaluate the robustness of the algorithm to measurement noise, we perturb the data with Gaussian white noise of intensity 0.01, expressed as  
\begin{align*}  
    \widetilde{u}(x,t) = u(x,t) + \sigma\xi(x,t), \quad t \in\{T/2,T\},  
\end{align*}  
where $\xi(\cdot,t)$ follows the standard normal distribution $N(0,1)$, and the noise intensity $\sigma = 0.01$. The regularization parameter $\beta$ in function \eqref{equ:discrete_loss_func} is chosen as 0.001. The inverse problem is solved by implementing the Adam optimizer with the following configuration: learning rate $\kappa = 0.01$, momentum parameters $\beta_1=0.9$ and $\beta_2=0.999$. All neural network parameters are initialized to zero before training. In the error analysis, we quantify the deviation between the inverted solution $y_{pred}(\cdot)$ and the ground truth $y_{ref}(\cdot)$ using both the absolute error defined by the $L^\infty$-norm $\| y_{ref}(\cdot) - y_{pred}(\cdot) \|_{L^\infty}$ and the relative error defined by the $L^2$-norm $\| y_{ref}(\cdot) - y_{pred}(\cdot) \|_{L^2}^2 / \| y_{ref}(\cdot) \|_{L^2}^2$. All computations were performed on a computing platform equipped with an NVIDIA GeForce RTX 2080 Ti GPU. The code used in this study is available at \url{https://github.com/279449621/DTPFI.git}.

\subsection{Allen-Cahn equation with logarithmic potential function}
\begin{figure}[htb]
    \centering
\includegraphics[width=0.95\linewidth]{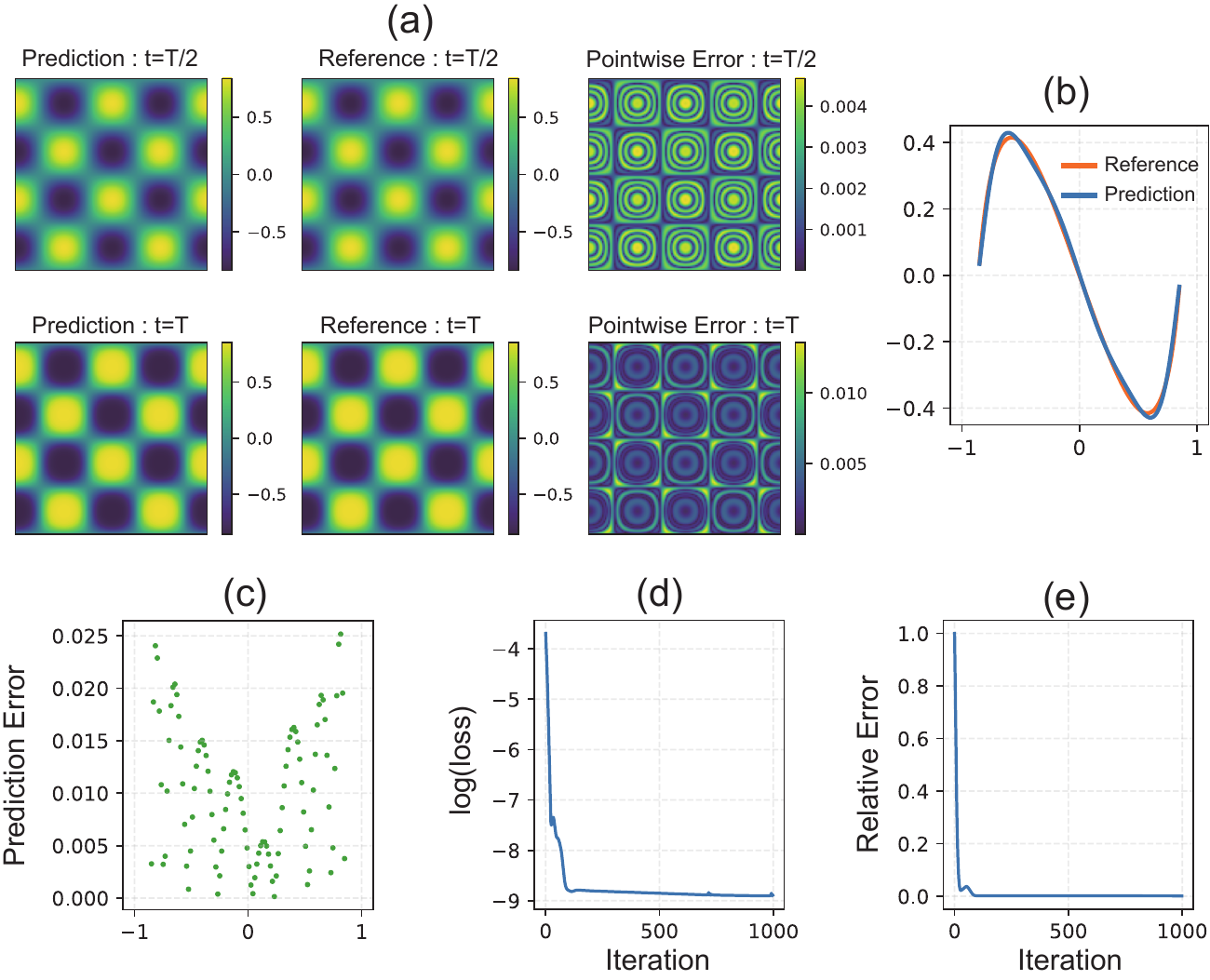}
    \caption{Inversion results for the Allen-Cahn equation. (a) Comparison of field evolution obtained using the reference and recovered coefficients at $t = T/2$ and $t = T$.
(b) Comparison between the reconstructed and reference potential functions on the interval $[-0.85, 0.85]$.
(c) Pointwise absolute error between the reconstructed and reference potential functions on the interval $[-0.85, 0.85]$.
(d) Evolution of the relative error between the predicted and reference fields at $t = T$ during optimization.
(e) Evolution of the relative error between the reconstructed and reference potential functions during optimization.}
    \label{fig:AC2D}
\end{figure}

This example considers the Allen-Cahn equation with a logarithmic potential function, specifically
\begin{equation}
    \left\{
    \begin{aligned}
        &\frac{\partial u}{\partial t} = 0.01^2\Delta u - \left(\ln\left(\frac{1+u}{1-u}\right) - 3u\right), &&\text{in}\ \Omega\times [0,T],\\
        &u(x,y,0) = 0.8\sin(4\pi x)\cos(4\pi y), &&\text{in}\ \Omega,
    \end{aligned}
    \right.
\end{equation}
where $\Omega = [0,1]^2$ and $T = 2$. 

As shown in Fig.~\ref{fig:AC2D}(a), the maximum absolute errors between the phase field evolutions generated by the reference and reconstructed potential functions are $4.7\times10^{-3}$ and $1.36\times10^{-2}$ at $T/2$ and $T$, respectively. This accuracy is consistent with the quality of the DTPFI reconstruction in the non-singular region $x \in (-0.85,0.85)$, where the pointwise errors remain below $2.5\times10^{-2}$ (Fig.~\ref{fig:AC2D}(b,c)). The optimization process also shows stable convergence: the discrepancy between the predicted ($\phi_{\mathrm{pred}}$) and reference ($\phi_{\mathrm{ref}}$) phase field profiles at time $T$ decreases monotonically over the iterations, while the relative error $\epsilon_{\mathrm{rel}}$ between the reconstructed and true potential functions rapidly drops to approximately $5\%$ within 20 iterations.

\subsection{Cahn-Hilliard equation with polynomial potential function}
\begin{figure}
    \centering
    \includegraphics[width=0.95\linewidth]{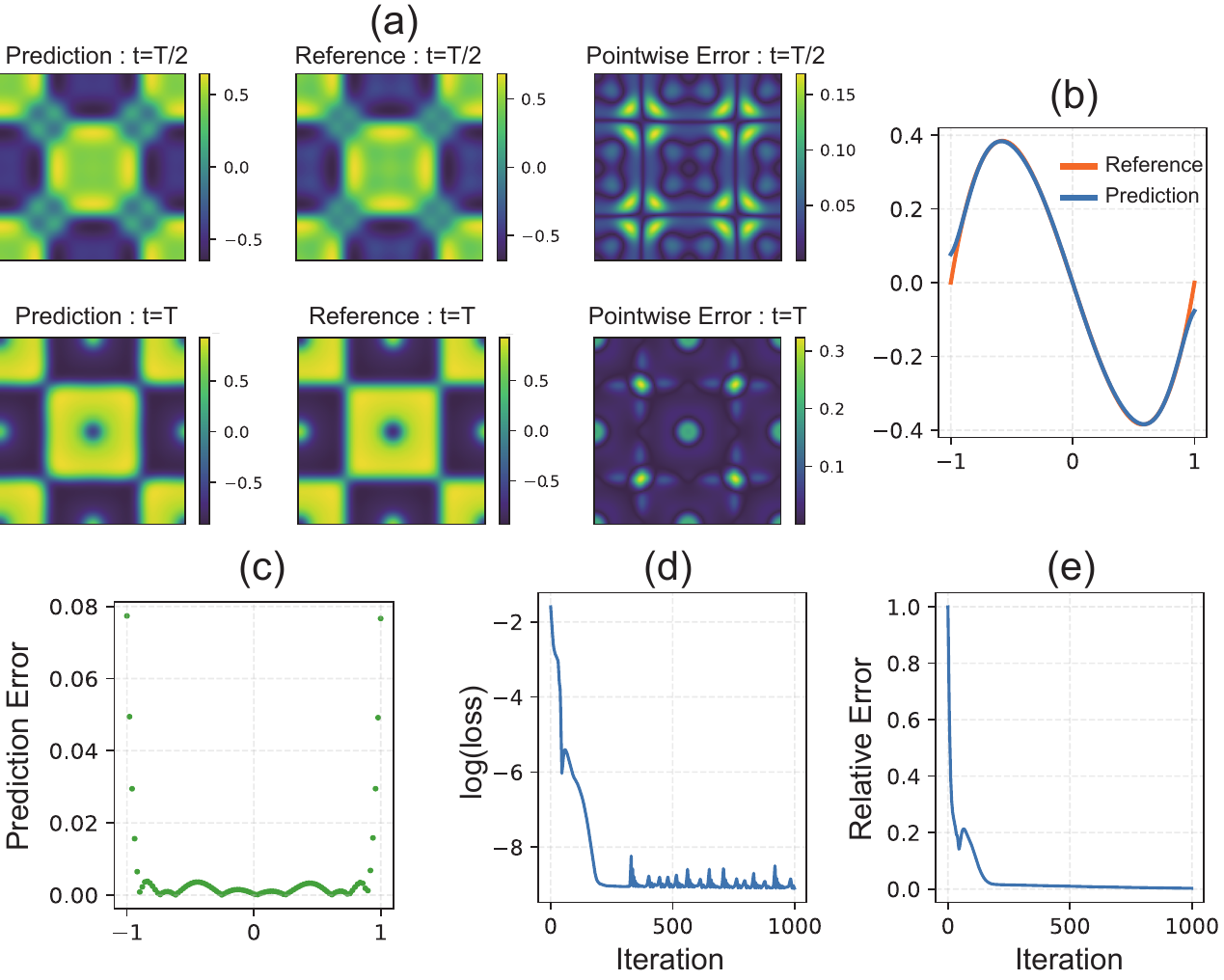}
    \caption{Inversion results for the Cahn-Hilliard equation. (a) Comparison of field evolution obtained using the reference and recovered coefficients at $t = T/2$ and $t = T$.
(b) Comparison between the reconstructed and reference potential functions.
(c) Pointwise absolute error between the reconstructed and reference potential functions.
(d) Evolution of the relative error between the predicted and reference fields at $t = T$ during optimization.
(e) Evolution of the relative error between the reconstructed and reference potential functions during optimization.}
    \label{fig:CH2D}
\end{figure}

This example considers the Cahn-Hilliard equation with a polynomial potential function, specifically
\begin{equation}
    \left\{
    \begin{aligned}
        &\frac{\partial u}{\partial t} = -0.02^2\Delta^2 u + \Delta (u^3 - u), &&\text{in}\ \Omega\times [0,T],\\
        &u(x,y,0) = 0.4\cos(3\pi x)\cos(3\pi y), &&\text{in}\ \Omega,
    \end{aligned}
    \right.
\end{equation}
where $\Omega = [0,1]^2$ and $T = 0.03$. 

As shown in Fig.~\ref{fig:CH2D}, the reconstructed chemical energy density function $f$ agrees well with the reference function. The maximum absolute error in the $L^{\infty}$ norm remains at $0.01$, which is comparable to the prescribed noise level, while the $L^2$ relative error reaches $1.7\times10^{-4}$ (Fig.~\ref{fig:CH2D}(c)). Phase field simulations performed with the recovered parameters are also consistent with the reference evolution. Although the maximum absolute errors at $T/2$ and $T$ are $2.46\times10^{-2}$ and $3.23\times10^{-1}$, respectively, the errors remain small over most of the spatial domain (Fig.~\ref{fig:CH2D}(a)), and the corresponding relative errors of the phase field are $2.96\times10^{-5}$ and $1.2\times10^{-2}$. In addition, the field error at $t=T$ decreases rapidly with the iteration number until convergence (Fig.~\ref{fig:CH2D}(d)), and the relative error between the reconstructed and reference potential functions falls below $5\%$ within 40 iterations (Fig.~\ref{fig:CH2D}(e)). These results indicate that DTPFI can accurately reconstruct polynomial chemical energy parameters using field measurements from only two time points.

\subsection{Sensitivity analysis}
This subsection examines how three key factors affect the performance of DTPFI: the truncation order of the basis expansion, the choice of basis functions, and the noise level in the measurement data.

\subsubsection{Truncation order}
We first evaluate the inversion accuracy for different truncation orders of the Chebyshev polynomial expansion. For the polynomial potential inversion in the Cahn-Hilliard equation, the maximum errors for the 5\textsuperscript{th}-, 10\textsuperscript{th}-, and 15\textsuperscript{th}-order truncations are $4.6\times10^{-3}$, $1.08\times10^{-2}$, and $2.9\times10^{-2}$, respectively (Fig.~\ref{fig:terms}(a)). The error increases as more terms are retained because the target chemical potential function $-x+x^3$ can already be represented exactly by the first four Chebyshev modes. Additional higher-order terms therefore introduce redundant parameters, which increase the optimization difficulty and may reduce inversion accuracy.

By contrast, for the logarithmic potential in the Allen-Cahn equation, the maximum errors on the interval $[-0.85,0.85]$ are $7.14\times10^{-2}$, $2.87\times10^{-2}$, and $1.93\times10^{-2}$ for the 5\textsuperscript{th}-, 10\textsuperscript{th}-, and 15\textsuperscript{th}-order truncations, respectively (Fig.~\ref{fig:terms}(b)). This decreasing trend indicates that, for singular logarithmic potentials, the Chebyshev approximation improves as more terms are retained.

Based on these observations, we use the 10\textsuperscript{th}-order truncation as a practical compromise: it provides sufficient accuracy for logarithmic potentials while avoiding the additional optimization burden introduced by excessive basis terms.

\begin{figure}
    \centering
\includegraphics[width=1\linewidth]{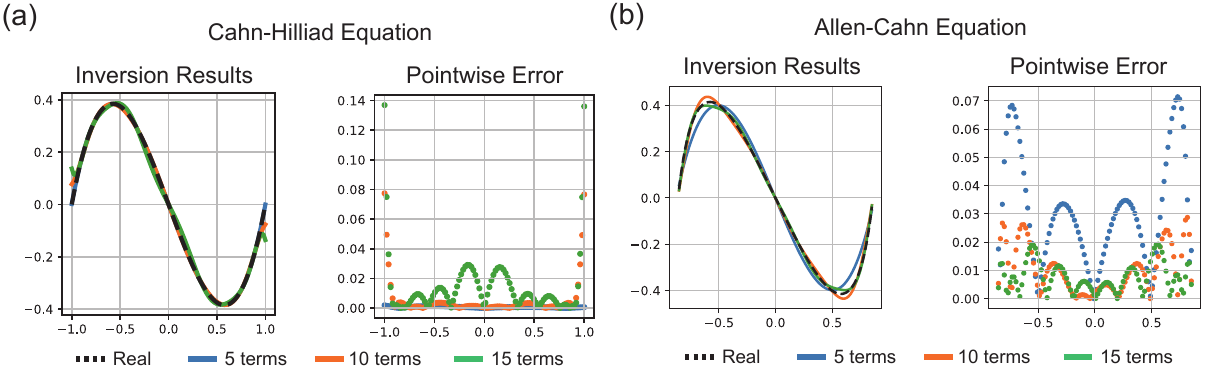}
    \caption{Inversion performance for different truncation orders. (a) Cahn-Hilliard equation: comparison between the reconstructed and reference potential functions (left) and the corresponding pointwise absolute errors (right). (b) Allen-Cahn equation: comparison between the reconstructed and reference potential functions (left) and the corresponding pointwise absolute errors (right).}
    \label{fig:terms}
\end{figure}

\subsubsection{Selection of basis functions}

\begin{table}[htbp]
  \centering
  \caption{Comparison of different basis functions in DTPFI.}
  \begin{tabular}{c *{3}{c} *{3}{c}}
    \toprule
    & \multicolumn{3}{c}{Cahn-Hilliard Equation} & \multicolumn{3}{c}{Allen-Cahn Equation} \\
    \cmidrule(lr){2-4} \cmidrule(lr){5-7}
    Basis & {Chebyshev} & {Monomial} & {Legendre} & {Chebyshev} & {Monomial} & {Legendre} \\
    \midrule
    Abs.Error & $1.08 \times 10^{-2}$ & $1.75 \times 10^{-1}$ & $1.26 \times 10^{-2}$ & $2.52 \times 10^{-2}$ & $4.07 \times 10^{-2}$ & $3.10 \times 10^{-2}$ \\
    Rel.Error & $1.74 \times 10^{-4}$ & $1.07 \times 10^{-1}$ & $7.27 \times 10^{-4}$ & $1.62 \times 10^{-3}$ & $4.49 \times 10^{-3}$ & $1.46 \times 10^{-3}$ \\
    \bottomrule
  \end{tabular}
  \label{tab:base}
\end{table}

We next compare the inversion performance obtained with three types of basis functions: Chebyshev, monomial, and Legendre bases. The corresponding errors are given in Table~\ref{tab:base}.

For both polynomial and logarithmic potential inversion, Chebyshev and Legendre polynomial bases yield higher accuracy than monomial bases. In particular, monomial basis functions perform substantially worse for the polynomial potential inversion. This degradation is caused by interference from higher-order monomial terms during training (Fig.~\ref{fig:Monomial}). In the early optimization stage, the 3rd, 5th, 7th, and 9th coefficients move in directions opposite to their reference values. These higher-order monomial terms introduce unnecessary degrees of freedom, which can lead to overfitting and reduce the accuracy of parameter recovery. This numerical instability is especially evident for polynomial potential inversion.

\subsubsection{Stability analysis with varying noise}
We finally compare the inversion performance under four noise levels, with noise intensities $\sigma=0$, $0.01$, $0.05$, and $0.1$.

\begin{figure}
    \centering
\includegraphics[width=0.95\linewidth]{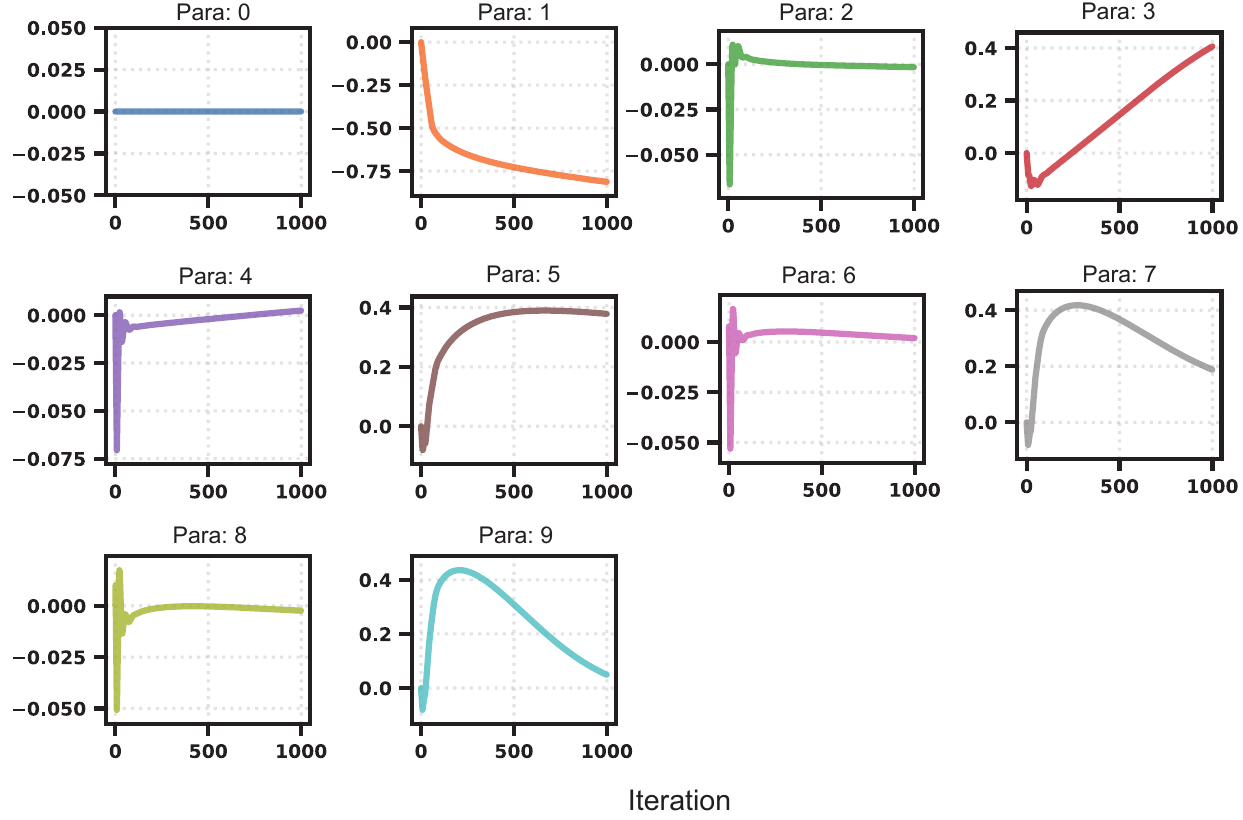}
    \caption{Evolution of monomial-basis coefficients during inversion of the Cahn-Hilliard equation.}
    \label{fig:Monomial}
\end{figure}

The inversion errors are summarized in Table~\ref{tab:noise}. Overall, DTPFI remains robust under noisy measurements. For example, when Gaussian noise with intensity $\sigma=0.01$ is added, the absolute error for the Cahn-Hilliard equation increases only from $1.01\times10^{-2}$ to $1.08\times10^{-2}$, while that for the Allen-Cahn equation changes from $2.40\times10^{-2}$ to $2.52\times10^{-2}$. These variations are less than 8\% and 5\%, respectively, indicating that the reconstruction accuracy remains close to the noise-free case under low noise levels. Even for larger noise intensities, the reconstruction errors remain controlled and comparable to the perturbation level, demonstrating the robustness of DTPFI with respect to measurement noise.

\begin{table}[h]
\centering
\caption{Inversion errors under different noise levels.}
\begin{tabular}{ccccc}
\toprule
\multicolumn{5}{c}{\textbf{Cahn-Hilliard Equation}} \\
\midrule
Noise Level & 0 & 0.01 & 0.05 & 0.10 \\
\cmidrule(r){1-1} \cmidrule(lr){2-2} \cmidrule(lr){3-3} \cmidrule(l){4-4}\cmidrule(l){5-5}
Abs.Error & $1.01\times10^{-2}$ & $1.08\times10^{-2}$ & $1.75\times10^{-1}$ & $1.26\times10^{-2}$ \\
\midrule
Rel.Error & $1.21\times10^{-4}$ & $1.74\times10^{-4}$ & $1.07\times10^{-1}$ & $7.27\times10^{-4}$ \\
\midrule
\multicolumn{5}{c}{\textbf{Allen-Cahn Equation}} \\
\midrule
Noise Level & 0 & 0.01 & 0.05 & 0.10 \\
\cmidrule(r){1-1} \cmidrule(lr){2-2} \cmidrule(lr){3-3} \cmidrule(l){4-4}\cmidrule(l){5-5}
Abs.Error & $2.40\times10^{-2}$ & $2.52\times10^{-2}$ & $4.07\times10^{-2}$ & $3.10\times10^{-2}$ \\
\midrule
Rel.Error & $1.57\times10^{-3}$ & $1.62\times10^{-3}$ & $4.49\times10^{-3}$ & $1.46\times10^{-3}$ \\
\bottomrule
\end{tabular}
\label{tab:noise}
\end{table}

\subsection{Extended applications}
In this subsection, we consider two more challenging inverse problems for phase field models. The first involves simultaneous recovery of the mobility and chemical potential functions in a Cahn-Hilliard equation with concentration-dependent mobility. The second concerns the identification of chemical energy parameters and an elastic-modulus related parameter in a coupled Cahn-Hilliard--Allen-Cahn phase field system.

\subsubsection{Cahn-Hilliard equation with variable mobility}
\begin{figure}  
\centering  
\includegraphics[width=0.95\linewidth]{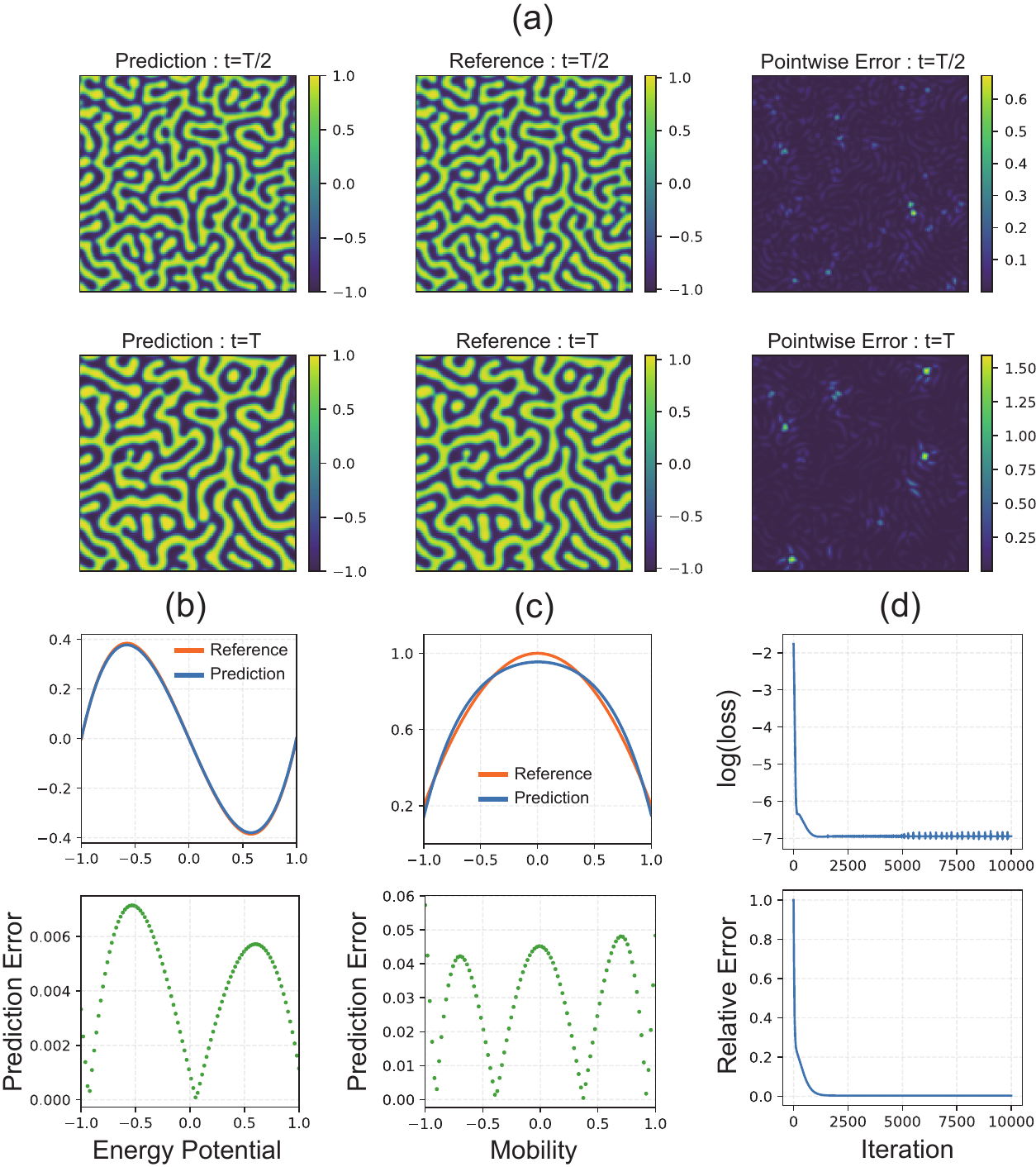}  
\caption{Inversion results for the potential function and field-dependent mobility in the Cahn-Hilliard equation. (a) Comparison of field evolution obtained using the reference and recovered coefficients at $t=T/2$ and $t=T$. (b) Comparison between the reconstructed and reference potential functions. (c) Comparison between the reconstructed and reference mobility functions. (d) Evolution of the relative error between the predicted and reference fields at $t=T$ during optimization (top) and between the reconstructed and reference potential functions (bottom).}  
\label{fig:CH_M}  
\end{figure}

Following \cite{zhu1999coarsening}, we consider the Cahn-Hilliard equation with concentration dependent mobility, given by  
\begin{equation}  
\left\{  
\begin{aligned}  
&\frac{\partial u}{\partial t} = \nabla\cdot \left[(1-0.8u^2)\nabla (u^3-u-0.5\Delta u)\right],&&\text{in}\ \Omega\times [0,T],\\  
&u(\cdot,0) = 0.1\nu, &&\text{in}\ \Omega,  
\end{aligned}  
\right.  
\end{equation}  
where $\nu(\cdot)$ is uniformly distributed in $[-1,1]$ over the domain $\Omega$, and the spatial and temporal step sizes are normalized to 1. The corresponding inverse problem is formulated as the following constrained optimization problem:  
\begin{align*}  
\min\limits_{\mathbf{a},\mathbf{b}} &\|u_{\mathbf{a},\mathbf{b}}(\cdot,t_2)-u(\cdot,t_2)\|_{L^2(\Omega)}^2+\beta (\|\mathbf{a}\|^2+\|\mathbf{b}\|^2), \\  
&\text{subject to: } u_{\mathbf{a},\mathbf{b}}(\cdot,t_1) = u(\cdot,t_1),  
\end{align*}  
Here, the unknown parameters are \(\mathbf{a}=(a_0,\cdots,a_4)\) and \(\mathbf{b}=(b_0,\cdots,b_4)\), and the state variable \(u_{\mathbf{a},\mathbf{b}}\) satisfies  
\begin{equation}  
\left\{  
\begin{aligned}  
&\frac{\partial u_{\mathbf{a},\mathbf{b}}}{\partial t} = \nabla\cdot \left[\left(\sum\limits_{n=0}^4 a_n\phi_n(u_{\mathbf{a},\mathbf{b}})\right)\nabla \left(\sum\limits_{n=0}^4 b_n\phi_n(u_{\mathbf{a},\mathbf{b}})-0.5 \Delta u_{\mathbf{a},\mathbf{b}}\right)\right],&&\text{in}\ \Omega\times [t_1,t_2],\\  
&u_{\mathbf{a},\mathbf{b}}(\cdot,0) = u(\cdot,t_1), &&\text{in}\ \Omega.  
\end{aligned}  
\right.  
\end{equation}

The inversion results in Fig.~\ref{fig:CH_M}(b,c) show that DTPFI accurately reconstructs both the potential function and the mobility function, with maximum absolute errors of $7\times 10^{-3}$ and $6\times 10^{-2}$, respectively. The larger error in the mobility reconstruction is attributable to the relatively weak influence of mobility on the phase field evolution. Nevertheless, the field evolution obtained from the recovered functions remains close to the reference solution, with relative errors of $5.07\times 10^{-2}$ and $1.45\times 10^{-1}$ at the two measurement times (Fig.~\ref{fig:CH_M}(a)). During optimization, the discrepancy between the predicted and reference fields at time \(T\) decreases rapidly at first, and both the potential and mobility functions approach a neighborhood of the reference solution (Fig.~\ref{fig:CH_M}(d)). At later iterations, the greater difficulty of recovering the mobility function to high precision leads to mild error oscillations, while the overall prediction error remains stable.

\subsubsection{Coupled CH-AC system with elastic energy}

We next employ the dimensionless coupled Cahn-Hilliard--Allen-Cahn equations from \cite{wang1995microstructural} to model phase evolution in Ni-Al binary alloys. The governing equations are:
\begin{equation}
\label{equ:phase_system}
\left\{
\begin{aligned}
&\frac{\partial c}{\partial t} = M\Delta \frac{\delta F}{\delta c}, &&\text{in}\ \Omega\times [0,T],\\
&\frac{\partial \eta}{\partial t} = -L\Delta \frac{\delta F}{\delta \eta},&&\text{in}\ \Omega\times [0,T],\\
&c(\cdot,0) = 0.6,&&\text{in}\ \Omega,\\
&\eta(\cdot,0) = 0.001\nu(\cdot),&&\text{in}\ \Omega,
\end{aligned}
\right.
\end{equation}
Here, $\nu$ is uniformly distributed in $[-1,1]$ over the computational domain $\Omega=[0,127]^2$, and the time interval is $t\in[0,40]$. The kinetic coefficients are set to $L=1$ and $M=0.4$. The total free energy functional $\mathcal{F}$ consists of chemical, gradient, and elastic energy contributions, given by:
\begin{align*}
&\mathcal{F} =\mathcal{F}_{gra}+\mathcal{F}_{chem}+\mathcal{F}_{elas},\\
&\mathcal{F}_{gra} = \int_\Omega \frac{\alpha}{2}|\nabla c|^2 + \frac{\beta}{2}|\nabla \eta|^2 \mathrm{d}V,\\
&\mathcal{F}_{chem} = \int_\Omega \frac{A}{2}(c-c_1)^2 + \frac{B}{2}(c_2-c)\eta^2 - \frac{C}{4}\eta^4 + \frac{D}{6}\eta^6 \mathrm{d}V,\\
&\mathcal{F}_{elas}= \frac{1}{2}\int \mathcal{B}(n) |\mathcal{F}(c)(k)|^2 \frac{\mathrm{d}^2k}{(2\pi)^2},
\end{align*}
where the chemical energy parameters are $A=7.5$, $B=4$, $C=1$, $D=0.5$, $c_1=0.1$, and $c_2=0.5$, and the gradient coefficients are $\alpha=\beta=1.7$. In the elastic energy term, $k$ denotes the wave vector, and the elastic interaction potential $\mathcal{B}(n)$ is defined as:
\begin{align*}
B(n) &= \mu\varepsilon_0^2 n_x^2n_y^2,\\
\mu &= \frac{4(C_{12}+2C_{44}-C_{11})(C_{11}+2C_{12})^2}{C_{11}(C_{11}+C_{12}+2C_{44})},
\end{align*}
where $n_x$ and $n_y$ are the components of the unit vector in frequency space, and $C_{11}$, $C_{12}$, and $C_{44}$ are elastic constants. In the experiments, we set $\mu\varepsilon_0^2=40$.

\begin{figure}
    \centering
    \includegraphics[width=0.95\linewidth]{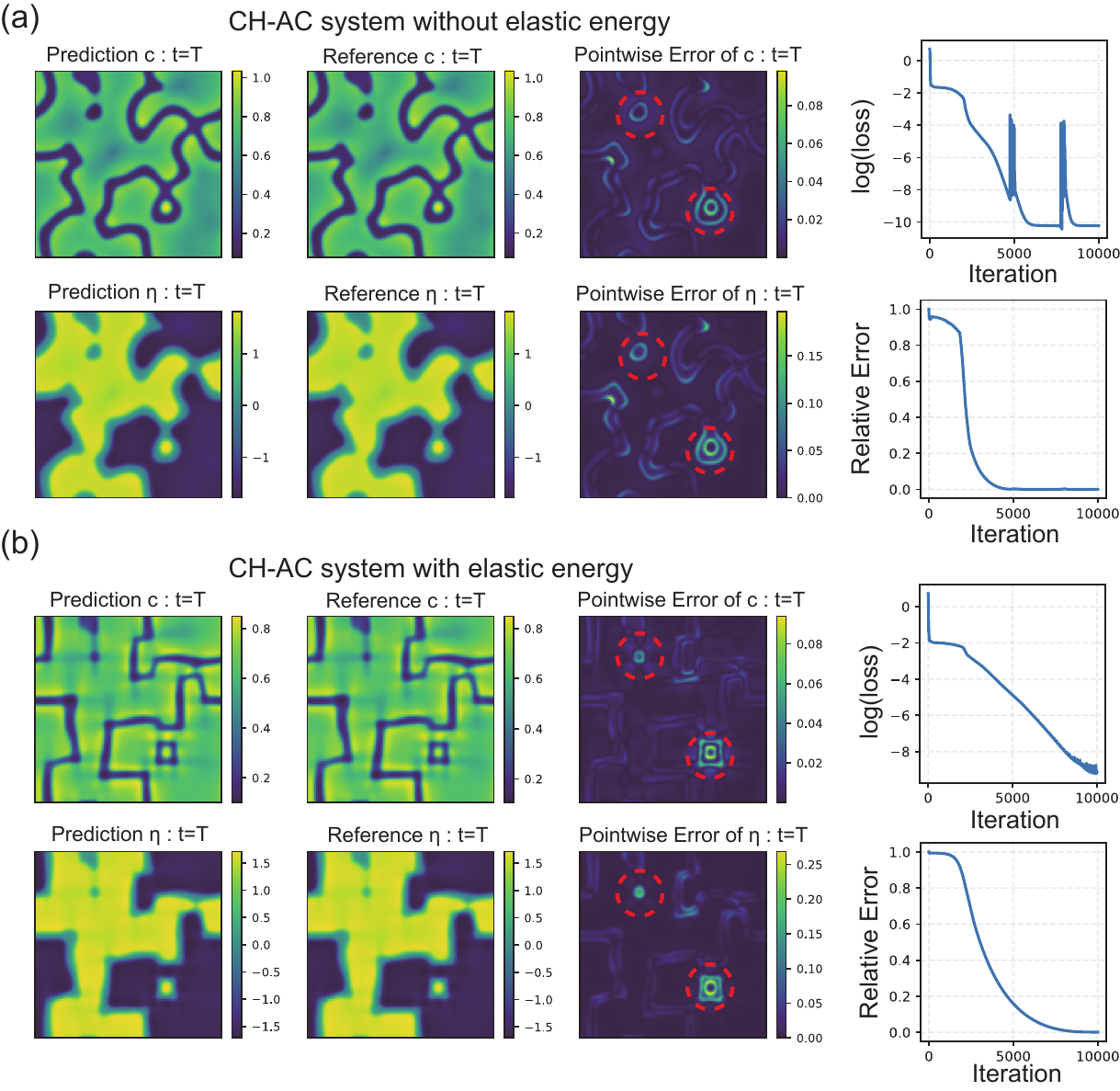}
    \caption{Comparison of inversion results for the CH-AC system. (a) Evolution of the concentration field (top) and order parameter field (bottom) at $t=T$ without elastic energy effects, comparing reference and reconstructed coefficients. (b) Evolution of the concentration field (top) and order parameter field (bottom) at $t=T$ with elastic energy effects, comparing reference and reconstructed coefficients.}
    \label{fig:ACCH2D}
\end{figure}

For the inverse problem, we consider two physical scenarios. The first neglects elastic energy effects and reconstructs only the chemical energy parameter set $\{A,B,C,D,c_1,c_2\}$. The second includes elastic energy effects and simultaneously reconstructs the chemical energy parameters together with the elastic-modulus-related parameter $\mu\varepsilon_0^2$. The input data consist of noisy measurements $\{c(\cdot,20),\eta(\cdot,20),c(\cdot,40),\eta(\cdot,40)\}$, contaminated with 1\% Gaussian white noise ($\sigma=0.01$). To enforce physical admissibility of the material parameters, we incorporate a non-negativity penalty and solve the optimization problem
\begin{align*}
    &\min\limits_{\mathbf{a}}\mathcal{J} = \min\limits_{\mathbf{a}}\|c_{\mathbf{a}}(\cdot,40)-c(\cdot,40)\|_2^2+\|\eta_{\mathbf{a}}(\cdot,40)-\eta(\cdot,40)\|_2^2+\beta\|\max(0,-\mathbf{a})\|_2^2,\\
    &s.t.\ c_{\mathbf{a}}(\cdot,20)=c_{obs}(\cdot,20), \eta_{\mathbf{a}}(\cdot,20)=\eta_{obs}(\cdot,20).
\end{align*}
Depending on whether elastic energy is included, the unknown coefficient vector is $\mathbf{a} \in \mathbb{R}^{6}$ or $\mathbf{a} \in \mathbb{R}^{7}$, and the regularization parameter is set to $\beta=10$. In the first scenario, where $\mathbf{a}=[a_1,\cdots,a_6]\in \mathbb{R}^6$, the concentration field $c_{\mathbf{a}}$ is governed by the chemical energy
\begin{align*}
    \mathcal{F}_{\text{chem},\mathbf{a}}=\int_{\Omega} \frac{a_1}{2}(c-a_5)^2+\frac{a_2}{2}(a_6-c)\eta^2-\frac{a_3}{4}\eta^4+\frac{a_4}{6}\eta^6\mathrm{d}V.
\end{align*}
When elastic energy effects are included, $c_{\mathbf{a}}$ is additionally governed by the elastic contribution
\begin{align*}
\mathcal{F}_{\text{elas},\mathbf{a}} =\frac{1}{2}\int a_7n_x^2n_y^2|\mathcal{F}(c)(k)|^2 \frac{\mathrm{d}^2k}{(2\pi)^2}.
\end{align*}
 
The inversion results for the two scenarios are given in Tables~\ref{tab:ACCH2D_NE} and~\ref{tab:ACCH2D}. Without elastic energy effects, the relative errors of all parameters except $c_1$ remain on the order of $10^{-3}$. When elastic energy is included, the parameters have substantially different magnitudes, yet most relative errors remain on the order of $10^{-2}$. The lower accuracy for $c_1$ has little influence on the predicted material evolution. As shown in Fig.~\ref{fig:ACCH2D}, the concentration and order parameter fields at the final time agree well between the reference and recovered parameters. Apart from minor discrepancies near interfaces and in small $\gamma$ and $\gamma'$ phase regions, marked by red dashed lines, the evolution patterns remain consistent across the domain. These results indicate that the recovered parameters can be used effectively for predictive phase field simulations.

\begin{table}[htbp]
  \centering
  \caption{Comparison of recovered coefficients and reference parameters for the CH-AC system without elastic energy effects.}
  \begin{tabular}{l *{6}{c}} 
    \toprule
    & {A} & {B} & {C} & {D} & {$c_1$} & {$c_2$} \\
    \midrule
    Reference    & 7.5    & 4      & 1      & 0.5    & 0.1    & 0.5   \\
    Predicted    & 7.55   & 4.02   & 1.00   & 0.50   & 0.17   & 0.50  \\
    \midrule
    Abs. error   & $4.82 \times 10^{-2}$ & $2.06 \times 10^{-2}$ & $2.68 \times 10^{-3}$ & $1.51 \times 10^{-5}$ & $7.11 \times 10^{-2}$ & $8.78 \times 10^{-4}$ \\
    Rel. error   & $6.43 \times 10^{-3}$ & $5.15 \times 10^{-3}$ & $2.68 \times 10^{-3}$ & $3.02 \times 10^{-5}$ & $7.11 \times 10^{-1}$ & $1.76 \times 10^{-3}$ \\
    \bottomrule
  \end{tabular}
   
       \label{tab:ACCH2D_NE}
\end{table}

\begin{table}[htbp]
  \centering
  
  \caption{Comparison of recovered coefficients and reference parameters for the CH-AC system with elastic energy effects.}
  \footnotesize
  \begin{tabular}{c *{6}{c} c} 
    \toprule
    & {A} & {B} & {C} & {D} & {$c_1$} & {$c_2$} & {$\mu\varepsilon_0^2$} \\
    \midrule
    Reference    & 7.5    & 4      & 1      & 0.5    & 0.1    & 0.5   & 40 \\
    Predicted    & 7.59   & 4.05   & 1.03   & 0.51   & 0.32   & 0.50  & 38.66 \\
    \midrule
    Abs. error   & $8.63 \times 10^{-2}$ & $4.83 \times 10^{-2}$ & $2.52 \times 10^{-2}$ & $1.23 \times 10^{-2}$ & $2.17 \times 10^{-1}$ & $3.24 \times 10^{-4}$ & 1.34 \\
    Rel. error   & $1.15 \times 10^{-2}$ & $1.21 \times 10^{-2}$ & $2.52 \times 10^{-2}$ & $2.46 \times 10^{-2}$ & $2.17 \times 10^{0}$ & $6.48 \times 10^{-4}$ & $3.34 \times 10^{-2}$ \\
    \bottomrule
  \end{tabular}
  
      \label{tab:ACCH2D}
\end{table}
\section{Conclusion}
In this work, we developed the Dual Time Phase Field Inversion (DTPFI) framework for recovering nonlinear potential functions in phase field systems. By truncating the basis expansion, the inverse potential problem is reformulated as a finite dimensional optimization problem. Theoretically, we established second order differentiability of the parameter to state map, proved the existence of minimizers, and derived local strong convexity of the objective functional. Computationally, we employed automatic differentiation for efficient gradient evaluation and demonstrated its equivalence to analytical sensitivity-based gradient computation. Numerical experiments show that DTPFI accurately reconstructs logarithmic potentials in the Allen-Cahn equation and polynomial potentials in the Cahn-Hilliard equation. The framework was further extended to field dependent mobility inversion, simultaneous mobility-potential recovery, and joint parameter identification in coupled Cahn-Hilliard-Allen-Cahn systems. Future work will focus on theoretically analyzing the stability of the inversion process and further applying the computational framework to three dimensional phase field models.

\bibliographystyle{unsrt} 
\bibliography{references} 

@article{cahn1958free,
  title={Free energy of a nonuniform system. I. Interfacial free energy},
  author={Cahn, John W and Hilliard, John E},
  journal={The Journal of chemical physics},
  volume={28},
  number={2},
  pages={258--267},
  year={1958},
  publisher={American Institute of Physics}
}

@article{allen1979microscopic,
  title={A microscopic theory for antiphase boundary motion and its application to antiphase domain coarsening},
  author={Allen, Samuel M and Cahn, John W},
  journal={Acta metallurgica},
  volume={27},
  number={6},
  pages={1085--1095},
  year={1979},
  publisher={Elsevier}
}

@article{brunk2023uniqueness,
  title={On uniqueness and stable estimation of multiple parameters in the Cahn-Hilliard equation},
  author={Brunk, Aaron and Egger, Herbert and Habrich, Oliver},
  journal={Inverse Problems},
  volume={39},
  number={6},
  pages={065002},
  year={2023},
  publisher={IOP Publishing}
}

@article{ni2024uniqueness,
  title={A uniqueness theory on determining the nonlinear energy potential in phase-field system},
  author={Ni, Tianhao and Lai, Jun},
  journal={Inverse Problems},
  volume={40},
  number={12},
  pages={125005},
  year={2024},
  publisher={IOP Publishing}
}

@article{kahle2020parameter,
  title={Parameter identification via optimal control for a Cahn--Hilliard-chemotaxis system with a variable mobility},
  author={Kahle, Christian and Lam, Kei Fong},
  journal={Applied Mathematics \& Optimization},
  volume={82},
  number={1},
  pages={63--104},
  year={2020},
  publisher={Springer}
}

@article{bao2025pfwnn,
  title={PFWNN: A deep learning method for solving forward and inverse problems of phase-field models},
  author={Bao, Gang and Ma, Chang and Gong, Yuxuan},
  journal={Journal of Computational Physics},
  pages={113799},
  year={2025},
  publisher={Elsevier}
}

@article{wang1995microstructural,
  title={Microstructural evolution during the precipitation of ordered intermetallics in multiparticle coherent systems},
  author={Wang, Yunzhi and Khachaturyan, Armen},
  journal={Philosophical Magazine A},
  volume={72},
  number={5},
  pages={1161--1171},
  year={1995},
  publisher={Taylor \& Francis}
}

@book{miranville2019cahn,
  title={The Cahn--Hilliard equation: recent advances and applications},
  author={Miranville, Alain},
  year={2019},
  publisher={SIAM}
}

@article{liu2012micro,
  title={Micro scale 3D FEM simulation on thermal evolution within the porous structure in selective laser sintering},
  author={Liu, FR and Zhang, Q and Zhou, WP and Zhao, JJ and Chen, JM},
  journal={Journal of Materials Processing Technology},
  volume={212},
  number={10},
  pages={2058--2065},
  year={2012},
  publisher={Elsevier}
}

@article{hu2007phase,
  title={Phase-field modeling of microvoid evolution under elastic-plastic deformation},
  author={Hu, SY and Baskes, MI and Stan, M},
  journal={Applied physics letters},
  volume={90},
  number={8},
  year={2007},
  publisher={AIP Publishing}
}

@article{sun2023phase,
  title={Phase field modeling of topological magnetic structures in ferromagnetic materials: domain wall, vortex, and skyrmion},
  author={Sun, Jiajun and Shi, Shengbin and Wang, Yu and Wang, Jie},
  journal={Acta Mechanica},
  volume={234},
  number={2},
  pages={283--311},
  year={2023},
  publisher={Springer}
}

@article{rubin1999three,
  title={Three-dimensional model of precipitation of ordered intermetallics},
  author={Rubin, G and Khachaturyan, AG},
  journal={Acta materialia},
  volume={47},
  number={7},
  pages={1995--2002},
  year={1999},
  publisher={Elsevier}
}

@book{khachaturyan2013theory,
  title={Theory of structural transformations in solids},
  author={Khachaturyan, Armen G},
  year={2013},
  publisher={Courier Corporation}
}

@article{lima2014hybrid,
  title={A hybrid ten-species phase-field model of tumor growth},
  author={Lima, EABF and Oden, JT and Almeida, RC},
  journal={Mathematical Models and Methods in Applied Sciences},
  volume={24},
  number={13},
  pages={2569--2599},
  year={2014},
  publisher={World Scientific}
}

@article{travasso2011phase,
  title={The phase-field model in tumor growth},
  author={Travasso, Rui DM and Castro, Mario and Oliveira, Joana CRE},
  journal={Philosophical Magazine},
  volume={91},
  number={1},
  pages={183--206},
  year={2011},
  publisher={Taylor \& Francis}
}

@article{wendler2009phase,
  title={Phase-field simulations of partial melts in geological materials},
  author={Wendler, Frank and Becker, Jens K and Nestler, Britta and Bons, Paul D and Walte, Nicolas P},
  journal={Computers \& geosciences},
  volume={35},
  number={9},
  pages={1907--1916},
  year={2009},
  publisher={Elsevier}
}

@article{mikelic2015phase,
  title={Phase-field modeling of a fluid-driven fracture in a poroelastic medium},
  author={Mikeli{\'c}, A and Wheeler, Mary F and Wick, Thomas},
  journal={Computational Geosciences},
  volume={19},
  pages={1171--1195},
  year={2015},
  publisher={Springer}
}

@book{ginzburg2009theory,
  title={On the theory of superconductivity},
  author={Ginzburg, Vitaly L and Ginzburg, Vitaly Lazarevich and Landau, LD},
  year={2009},
  publisher={Springer}
}

@article{steinbach2009phase,
  title={Phase-field models in materials science},
  author={Steinbach, Ingo},
  journal={Modelling and simulation in materials science and engineering},
  volume={17},
  number={7},
  pages={073001},
  year={2009},
  publisher={IOP Publishing}
}

@article{chen2002phase,
  title={Phase-field models for microstructure evolution},
  author={Chen, Long-Qing},
  journal={Annual review of materials research},
  volume={32},
  number={1},
  pages={113--140},
  year={2002},
  publisher={Annual Reviews 4139 El Camino Way, PO Box 10139, Palo Alto, CA 94303-0139, USA}
}

@article{zhu1999coarsening,
  title={Coarsening kinetics from a variable-mobility Cahn-Hilliard equation: Application of a semi-implicit Fourier spectral method},
  author={Zhu, Jingzhi and Chen, Long-Qing and Shen, Jie and Tikare, Veena},
  journal={Physical Review E},
  volume={60},
  number={4},
  pages={3564},
  year={1999},
  publisher={APS}
}

@article{kingma2014adam,
  title={Adam: A method for stochastic optimization},
  author={Kingma, Diederik P and Ba, Jimmy},
  journal={arXiv preprint arXiv:1412.6980},
  year={2014}
}

@article{gopfert1973mathematische,
  title={Mathematische Optimierung in allgemeinen Vektorraumen},
  author={Gopfert, A},
  journal={Teubner, Leipzig},
  year={1973}
}

@article{cherfils2000generalized,
  title={Generalized Cahn-Hilliard equations with a logarithmic free energy.},
  author={Cherfils, L and Miranville, A},
  journal={Revista de la Real Academia de Ciencias Exactas, F{\'\i}sicas y Naturales},
  volume={94},
  number={1},
  pages={19--32},
  year={2000},
  publisher={Real Academia de Ciencias Exactas, F{\'\i}sicas y Naturales}
}

@article{elliott1989nonconforming,
  title={A nonconforming finite-element method for the two-dimensional Cahn--Hilliard equation},
  author={Elliott, Charles M and French, Donald A},
  journal={SIAM Journal on Numerical Analysis},
  volume={26},
  number={4},
  pages={884--903},
  year={1989},
  publisher={SIAM}
}

@book{evans2022partial,
  title={Partial differential equations},
  author={Evans, Lawrence C},
  volume={19},
  year={2022},
  publisher={American mathematical society}
}

@article{morocsanu2016well,
  title={Well-posedness for a phase-field transition system endowed with a polynomial nonlinearity and a general class of nonlinear dynamic boundary conditions},
  author={Moro{\c{s}}anu, Costic{\u{a}}},
  journal={Journal of Fixed Point Theory and Applications},
  volume={18},
  number={2},
  pages={225--250},
  year={2016},
  publisher={Springer}
}

@article{barrett2002finite,
  title={Finite element approximation of a degenerate Allen--Cahn/Cahn--Hilliard system},
  author={Barrett, John W and Blowey, James F},
  journal={SIAM Journal on Numerical Analysis},
  volume={39},
  number={5},
  pages={1598--1624},
  year={2002},
  publisher={SIAM}
}

@article{barrett2001fully,
  title={On fully practical finite element approximationsof degenerate Cahn-Hilliard systems},
  author={Barrett, John W and Blowey, James F and Garcke, Harald},
  journal={ESAIM: Mathematical Modelling and Numerical Analysis},
  volume={35},
  number={4},
  pages={713--748},
  year={2001},
  publisher={EDP Sciences}
}

@article{brunk2021analysis,
  title={Analysis of a viscoelastic phase separation model},
  author={Brunk, Aaron and D{\"u}nweg, Burkhard and Egger, Herbert and Habrich, Oliver and Luk{\'a}{\v{c}}ov{\'a}-Medvid'ov{\'a}, M{\'a}ria and Spiller, Dominic},
  journal={Journal of Physics: Condensed Matter},
  volume={33},
  number={23},
  pages={234002},
  year={2021},
  publisher={IOP Publishing}
}

@article{feng2006fully,
  title={Fully Discrete Finite Element Approximations of the Navier--Stokes--Cahn-Hilliard Diffuse Interface Model for Two-Phase Fluid Flows},
  author={Feng, Xiaobing},
  journal={SIAM journal on numerical analysis},
  volume={44},
  number={3},
  pages={1049--1072},
  year={2006},
  publisher={SIAM}
}

@article{chen1998applications,
  title={Applications of semi-implicit Fourier-spectral method to phase field equations},
  author={Chen, Long Qing and Shen, Jie},
  journal={Computer Physics Communications},
  volume={108},
  number={2-3},
  pages={147--158},
  year={1998},
  publisher={Elsevier}
}

@article{shen2010numerical,
  title={Numerical approximations of allen-cahn and cahn-hilliard equations},
  author={Shen, Jie and Yang, Xiaofeng},
  journal={Discrete Contin. Dyn. Syst},
  volume={28},
  number={4},
  pages={1669--1691},
  year={2010}
}

@article{eyre1998unconditionally,
  title={Unconditionally gradient stable time marching the Cahn-Hilliard equation},
  author={Eyre, David J},
  journal={MRS online proceedings library (OPL)},
  volume={529},
  pages={39},
  year={1998},
  publisher={Cambridge University Press}
}

@article{fu2022energy,
  title={Energy-decreasing exponential time differencing Runge--Kutta methods for phase-field models},
  author={Fu, Zhaohui and Yang, Jiang},
  journal={Journal of Computational Physics},
  volume={454},
  pages={110943},
  year={2022},
  publisher={Elsevier}
}

@article{shen2018scalar,
  title={The scalar auxiliary variable (SAV) approach for gradient flows},
  author={Shen, Jie and Xu, Jie and Yang, Jiang},
  journal={Journal of Computational Physics},
  volume={353},
  pages={407--416},
  year={2018},
  publisher={Elsevier}
}

@article{kaltenbacher2025reconstruction,
  title={Reconstruction of space-dependence and nonlinearity of a reaction term in a subdiffusion equation},
  author={Kaltenbacher, Barbara and Rundell, William},
  journal={Inverse Problems},
  volume={41},
  number={5},
  pages={055008},
  year={2025},
  publisher={IOP Publishing}
}

@article{kaltenbacher2026identification,
  title={Identification of space-dependent coefficients in two competing terms of a nonlinear subdiffusion equation},
  author={Kaltenbacher, Barbara and Rundell, William},
  journal={arXiv preprint arXiv:2601.21018},
  year={2026}
}

@article{jin2021error,
  title={Error analysis of finite element approximations of diffusion coefficient identification for elliptic and parabolic problems},
  author={Jin, Bangti and Zhou, Zhi},
  journal={SIAM Journal on Numerical Analysis},
  volume={59},
  number={1},
  pages={119--142},
  year={2021},
  publisher={SIAM}
}

@article{jin2023convergence,
  title={Convergence rate analysis of Galerkin approximation of inverse potential problem},
  author={Jin, Bangti and Lu, Xiliang and Quan, Qimeng and Zhou, Zhi},
  journal={Inverse Problems},
  volume={39},
  number={1},
  pages={015008},
  year={2023},
  publisher={IOP Publishing}
}

@article{zhang2022identification,
  title={Identification of potential in diffusion equations from terminal observation: analysis and discrete approximation},
  author={Zhang, Zhengqi and Zhang, Zhidong and Zhou, Zhi},
  journal={SIAM Journal on Numerical Analysis},
  volume={60},
  number={5},
  pages={2834--2865},
  year={2022},
  publisher={SIAM}
}

@article{bertacco2021stochastic,
  title={Stochastic Allen--Cahn equation with logarithmic potential},
  author={Bertacco, Federico},
  journal={Nonlinear Analysis},
  volume={202},
  pages={112122},
  year={2021},
  publisher={Elsevier}
}

@book{ladyzhenskaia1968linear,
  title={Linear and quasi-linear equations of parabolic type},
  author={Ladyzhenskaia, Olga Aleksandrovna and Solonnikov, Vsevolod Alekseevich and Ural'tseva, Nina N},
  volume={23},
  year={1968},
  publisher={American Mathematical Soc.}
}

@article{yuan2022pinn,
  title={A-PINN: Auxiliary physics informed neural networks for forward and inverse problems of nonlinear integro-differential equations},
  author={Yuan, Lei and Ni, Yi-Qing and Deng, Xiang-Yun and Hao, Shuo},
  journal={Journal of Computational Physics},
  volume={462},
  pages={111260},
  year={2022},
  publisher={Elsevier}
}

@article{cherfils2011cahn,
  title={The Cahn-Hilliard equation with logarithmic potentials},
  author={Cherfils, Laurence and Miranville, Alain and Zelik, Sergey},
  journal={Milan Journal of Mathematics},
  volume={79},
  number={2},
  pages={561--596},
  year={2011},
  publisher={Springer}
}
\end{document}